\documentclass{amsart}
\usepackage{graphicx}

\usepackage{amsfonts}
\usepackage{amssymb}
\usepackage{amsmath}
\usepackage{amsthm}
\usepackage{amscd}

\def\vbar{\mathchoice{\vrule height6.3ptdepth-.5ptwidth.8pt\kern- .8pt}
{\vrule height6.3ptdepth-.5ptwidth.8pt\kern-.8pt} {\vrule
height4.1ptdepth-.35ptwidth.6pt\kern-.6pt} {\vrule
height3.1ptdepth-.25ptwidth.5pt\kern-.5pt}}
\def\fudge{\mathchoice{}{}{\mkern.5mu}{\mkern.8mu}}
\def\bbc#1#2{{\rm \mkern#2mu\vbar\mkern-#2mu#1}}
\def\bbb#1{{\rm I\mkern-3.5mu #1}}
\def\bba#1#2{{\rm #1\mkern-#2mu\fudge #1}}
\def\bb#1{{\count4=`#1 \advance\count4by-64 \ifcase\count4\or\bba
A{11.5}\or \bbb B\or\bbc C{5}\or\bbb D\or\bbb E\or\bbb F \or\bbc
G{5}\or\bbb H\or \bbb I\or\bbc J{3}\or\bbb K\or\bbb L \or\bbb
M\or\bbb N\or\bbc O{5} \or \bbb P\or\bbc C{5}\or\bbb B\or\bbc
S{4.2}\or\bba T{10.5}\or\bbc U{5}\or \bba V{12}\or\bba
W{16.5}\or\bba X{11}\or\bba Y{11.7}\or\bba Z{7.5}\fi}}

\newcommand{\K}{{\mathbb{K}}}

\newtheorem{df}{Definition}[section]
\newtheorem{thm}{Theorem}[section]
\newtheorem{cor}{Corollary}[section]
\newtheorem{rem}{Remark}[section]

\newtheorem{prop}{Proposition}[section]
\newtheorem{exa}{Example}[section]
\newtheorem{lem}{Lemma}[section]

\begin{document}

\date{}
\title{Rota-Baxter Operators on  pre-Lie superalgebras and beyond}
\author{K. Abdaoui, S. Mabrouk and A. Makhlouf }
\address{El-Kadri  Abdaoui , Universit\'{e} de Sfax, Facult\'{e} des Sciences Sfax,  BP
1171, 3038 Sfax, Tunisia.}
\address{Sami Mabrouk, Universit\'{e} de Gafsa,  Facult\'{e} des Sciences, Gafsa Tunisia}%
\address{Abdenacer Makhlouf, Universit\'{e} de Haute Alsace, 4 rue des fr\`eres Lumi\`ere, 68093 Mulhouse France.}
\email{Abdaouielkadri@hotmail.com}
\email{Mabrouksami00@yahoo.fr}
\email{Abdenacer.Makhlouf@uha.fr}
 \maketitle{}
 %\keywords{Classification of Rota-Baxter on some pre-Lie superalgebras}%
 % ----------------------------------------------------------------

 \begin{abstract}
 In this paper, we study Rota-Baxter operators and super $\mathcal{O}$-operator of associative superalgebras, Lie superalgebras, pre-Lie superalgebras and $L$-dendriform superalgebras. Then we give some properties of pre-Lie superalgebras constructed from associative superalgebras, Lie superalgebras and $L$-dendriform superalgebras. Moreover,  we provide all Rota-Baxter operators of weight zero on  complex pre-Lie superalgebras of dimensions $2$ and $3$.
 \end{abstract}
\section*{Introduction}
Rota-Baxter operators of weight $\lambda \in \mathbb{K}$ fulfil the so-called Rota-Baxter relation which may be regarded as one possible generalization of the standard shuffle relation \cite{Guo-Keigher shuffleprod,GC-Rota1998}. They appeared for the first time in the work of the mathematician G. Baxter \cite{Baxter1960} in $1960$ and were then intensively studied by F. V. Atkinson \cite{Atkinson1967}, J. B. Miller \cite{Miller1969}, G.-C. Rota \cite{Rota1995}, P. Cartier \cite{Cartier1972} and  more recently they reappeared in the work of L. Guo \cite{Guo-Introd} and K. Ebrahimi-Fard \cite{Ebrahimi-loday.algebras}. \\
Pre-Lie algebras $($called also left-symmetric algebras, Vinberg algebras, quasi-associative algebras$)$ are a class of a natural algebraic systems appearing in many fields in mathematics and mathematical physics. They were first mentioned by A. Cayley in $1890$ \cite{Cayley} as a kind of rooted tree algebra and later arose again from the study of convex homogeneous cones \cite{Vinberg}, affine manifold and affine structures on Lie groups \cite{Koszul}, and deformation of associative algebras \cite{Gerstenhaber}. They play an important role in the study of symplectic and complex structures on Lie groups and Lie algebras \cite{Andrada-Salamon,Chu,Dardié-Medina1,Dardié-Medina2,Lichnerowicz-Medina}, phases spaces of Lie algebras \cite{Bai phase spaces,Kupershmidt2}, certain integrable systems \cite{Bordemann}, classical and quantum Yang-Baxter equations \cite{Diata-Medina}, combinatorics \cite{Ebrahimi-loday.algebras}, quantum field theory \cite{Connes-Kreimer}, and operads \cite{Chapoton-Livernet}. See \cite{Burde} for  a survey. Recently, pre-Lie superalgebras, the  $\mathbb{Z}_2$-graded version of pre-Lie algebras, also appeared in many others fields; see for example \cite{Chapoton-Livernet,Gerstenhaber,Mikhalev}. To our knowledge, they were first introduced by Gerstenhaber in $1963$ to study the cohomology structure of associative algebras \cite{Gerstenhaber}. They are a class of natural algebraic appearing in many fields in mathematics and mathematical physics, especially in  super-symplectic geometry, vertex superalgebras and graded classical Yang-Baxter equation. Recently, the classifications of complex pre-Lie superalgebras in dimensions two and three were given by R. Zhang and C.M.  Bai \cite{Bai and Zhang classif}. \\
 It turns out that the construction of pre-Lie superalgebras from associative superalgebras uses Rota-Baxter operators. Let $\mathcal{A}$ be an associative superalgebra (product of $x$ and $y$ is denoted by $xy$) and $R$ be a Rota-Baxter operator of weight $\lambda$ on $\mathcal{A}$, which means that it satisfies, for any homogeneous elements $x,y$ in $\mathcal{A}$, the identity
\begin{equation}\label{R-B assoc}
R(x)R(y)=R\Big( R(x)y+xR(y)+\lambda xy\Big).
\end{equation}
If $\lambda=0$ (resp. $\lambda=-1$),  the product
\begin{equation}\label{intr-ass=pre-Lie0}
    x\circ y=R(x)y-(-1)^{|x||y|}y R(x),~~\forall~~x,y\in\mathcal{H}(\mathcal{A})
\end{equation}
resp.
\begin{equation}\label{intr-ass=pre-Lie}
    x\circ y=R(x)y-(-1)^{|x||y|}y R(x)-xy,~~\forall~~x,y\in\mathcal{H}(\mathcal{A})
\end{equation}
defines a pre-Lie superalgebra $($see Theorem \ref{RB+ass=preLie super}$)$.\\

The notion of dendriform algebras was introduced  in $1995$ by J.-L.  Loday \cite{Loday-dialgebras}. Dendriform algebras are algebras with two operations, which dichotomize the notion of associative algebras. The motivation came from algebraic $\mathbb{K}$-theory, they  have been studied quite extensively with connections to several areas in mathematics and physics, including operads, homology, Hopf algebras, Lie and Leibniz algebras, combinatorics, arithmetic and quantum field theory and so on $($see \cite{Ebrahimi-Manchon-Patr} and the references therein$)$.
The relationship between dendriform algebras, Rota-Baxter algebras and pre-Lie algebras was given by M. Aguiar and K. Ebrahimi-Fard \cite{Aguiar2000,Ebrahimi-loday.algebras,Ebrahimi-assoc-Nijenhuis}. C. Bai, L. Liu, L. Guo and X. Ni, generalized the concept of Rota-Baxter operator and introduced a new class of algebras, namely, $L$-dendriform algebras, in \cite{ Bai-Liu OOperaorLie,Bai-Liu OOperaorAss,Bai-Liu L-dendrifom}.
Moreover, a close relationship among associative superalgebras, Lie superalgebras, pre-Lie superalgebras and dendriform superalgebras is given as follows in the sense of commutative diagram of categories:
$$\begin{array}{ccc}
  \text{Lie superalgebra} & \longleftarrow & \text{pre-Lie superalgebra}\\
  \uparrow &  & \uparrow \\
  \text{associative superalgebra} & \longleftarrow & \text{dendriform superalgebra}
\end{array}$$
Recently, the notion of Rota-Baxter operator on a bimodules was introduced by M. Aguiar \cite{Aguiar2004}. The construction of associative, Lie, pre-Lie and $L$-dendriform superalgebras are extended to the corresponding categories of bimodules.\\
The main purpose of this paper is to study, through Rota-Baxter operators and $\mathcal{O}$-operators,  the relationship between associative superalgebras, Lie superalgebras, pre-Lie superalgebras and $L$-dendriform superalgebras. Moreover, we classify  Rota-Baxter operators of weight zero on the complex pre-Lie superalgebras of dimensions $2$ and $3$.\\  This paper is organized as follows.
In Section $1$, we recall some definitions of associative superalgebras, Lie superalgebras and pre-Lie superalgebras and we introduce the notion of super $\mathcal{O}$-operator of these superalgebras that generalizes the notion of Rota-Baxter operators. We show that every Rota-Baxter associative superalgebra of weight $\lambda=-1$ gives rise to a Rota-Baxter Lie superalgebra. Moreover, a super $\mathcal{O}$-operator on  a Lie superalgebra $($of weight zero$)$ gives rise to  a pre-Lie superalgebra. As Example of computations,   we provide  all Rota-Baxter operators $($of weight zero$)$ on the orthosymplectic Lie superalgebra $osp(1,2)$. In Section $2$, we introduce the notion of $L$-dendriform superalgebra and then study some fundamental properties of $L$-dendriform superalgebras in terms of  super $\mathcal{O}$-operator of pre-Lie superalgebras. Their relationship with associative superalgebras are also described. Sections $3$ and $4$  are devoted to classification of all Rota-Baxter operators $($of weight zero$)$ on the complex pre-Lie superalgebras of dimension $2$ and $3$ with $1$-dimensional even part and with $2$-dimensional even part respectively.\\
Throughout this paper, all superalgebras are finite-dimensional and are over a field $\mathbb{K}$ of characteristic zero.  Let $(\mathcal{A},\circ)$ be a superalgebra, then $L_\circ$ and $R_\circ$ denote the even  left and right multiplication operators $L_\circ,R_\circ:\mathcal{A}\rightarrow End(\mathcal{A})$ defined as  $L_\circ(x)(y)=(-1)^{|x||y|}R_\circ(y)(x)=x\circ y$ for all homogeneous element $x,y$ in $\mathcal{A}$.
 In particular, when $(\mathcal{A},[ \ ,\ ])$ is a Lie superalgebra, we let $ad(x)$ denote the adjoint operator, that is, $ad(x)(y)=[x,y]$ for all homogeneous element $x,y$ in $\mathcal{A}$.

 \section{ Rota-Baxter associative superalgebras, pre-Lie superalgebras and Lie superalgebras}
Let $(\mathcal{A},\circ)$ be an algebra over a field $\mathbb{K}$. It is said to be a superalgebra if the underlying vector space of $\mathcal{A}$ is $\mathbb{Z}_2$-graded, that is, $\mathcal{A}=\mathcal{A}_0\oplus \mathcal{A}_1$, and $\mathcal{A}_i\circ \mathcal{A}_j\subset \mathcal{A}_{i+j}$, for $i,j\in \mathbb{Z}_2$. An element of $\mathcal{A}_0$ is said to be  even and an element of $\mathcal{A}_1$ is said to be odd. The elements of $\mathcal{A}_{j},~~j \in \mathbb{Z}_2$, are said to be homogenous and of parity $j$. The parity of a homogeneous element $x$ is denoted by $|x|$ and we refer to the set of homogeneous elements of $\mathcal{A}$ by $\mathcal{H}(\mathcal{A})$.\\

We extend to graded case the concepts of  $\mathcal{A}$-bimodule $\K$-algebra,  $\mathcal{O}$-operator and extended  $\mathcal{O}$-operator introduced in \cite{Bai-Liu OOperaorAss}.
\begin{df}\ \begin{enumerate}
\item An associative superalgebra is a pair $(\mathcal{A},\mu)$ consisting of a $\mathbb{Z}_2$-graded vector space $\mathcal{A}$
and an even bilinear map $\mu:\mathcal{A}\otimes \mathcal{A}\longrightarrow \mathcal{A},~~(\mathcal{A}_i\mathcal{A}_j\subseteq \mathcal{A}_{i+j},~~\forall~~i,j\in \mathbb{Z}_2) $ satisfying for all $x,y,z\in\mathcal{H}(\mathcal{A})$
$$x(y z)=(x y)z.$$
\item Let $(\mathcal{A},\mu)$ be an associative superalgebra and $V$ be a $\mathbb{Z}_2$-graded vector space. Let $l,r:\mathcal{A}\longrightarrow End(V)$ be two even linear maps.  A triple $(V,l,r)$ is called an $\mathcal{A}$-bimodule if for all  $x,y \in \mathcal{H}(\mathcal{A})$ and $v \in \mathcal{H}(V)$
$$ l(xy)(v)=l(x)l(y)(v),~~r(xy)(v)=r(y)r(x)(v),~~l(x)r(y)(v)=r(y)l(x)(v).$$
%$$ l(xy)(v)=l(x)l(y)(v),~~r(xy)(v)=(-1)^{|x||y|}r(y)r(x)(v),~~l(x)r(y)(v)=(-1)^{|x||y|}r(y)l(x)(v).$$
Moreover, the quadruple $(V,\mu_V,l,r)$ is said to be an  $\mathcal{A}$-bimodule $\K$-superalgebra if $(V,l,r)$  is an  $\mathcal{A}$-bimodule compatible with the multiplication $\mu_V$ on $V$, that is, for all $x,y \in\mathcal{H}(\mathcal{A})$ and $v,w\in \mathcal{H}(V),$ 
\begin{eqnarray*}&  l(x)(\mu_V(v,w))=\mu_V(l(x)(v),w), ~~ r(x)(\mu_V(v,w))= \mu_V(v,r(x)(w)), \\& ~~ \mu_V(r(x)(v),w)=\mu_V(v,l(x)(w)).
\end{eqnarray*}

 \item Fix $\lambda\in \K$,  A pair $(T,T')$ of even linear maps $T,T':V\longrightarrow \mathcal{A}$ is called an extended super $\mathcal{O}$-operator with modification $T'$ of
weight $\lambda$  associated to the bimodule $(V,l,r)$ if $T$ satisfies
\begin{eqnarray}
& \lambda l(T'(u))v=\lambda r(T'(v))u,\\
&  T(u)T(v)=T\Big(l(T(u))v+(-1)^{|u||v|}r(T(v))u\Big)+\lambda T'(u)T'(v),~~\forall~~u,v \in \mathcal{H}(V). 
 \end{eqnarray}
 \item  An even  linear map $T:V\longrightarrow \mathcal{A}$ is called a super $\mathcal{O}$-operator of
weight $\lambda$  associated to the bimodule $\K$-superalgebra  $(V,\mu_V,l,r)$ if it satisfies
\begin{eqnarray}
  T(u)T(v)=T\Big(l(T(u))v+(-1)^{|u||v|}r(T(v))u+\lambda \mu_V(u,v)\Big),~~\forall~~u,v \in \mathcal{H}(V). 
 \end{eqnarray}
\end{enumerate}
\end{df}
Notice that the notions of  super $\mathcal{O}$-operator and extended  super $\mathcal{O}$-operator coincide when $\lambda =0$.

%When $T' = 0$  or $\lambda = 0$, we obtain the concept of a super  $\mathcal{O}$-operator $T$ satisfying
%$$ T(u)  T(v) = T(l(T(u))v) + T(r(T(v))u), \  \forall u, v \in V.$$

%??? When $V$ is taken to be the $A$-bimodule $(A, L,R)$ where $ L,R : A \rightarrow End_\K(A$) are given by
%the left and right multiplications, an $\mathcal{O}$-operator $T : V \rightarrow A$ of weight zero is just a Rota-Baxter operator of weight zero.

In particular, a super $\mathcal{O}$-operator of weight $\lambda\in \mathbb{K} $ associated to the bimodule $\K$-algebra $(\mathcal{A},\mu_A,L_\mu,R_\mu)$ is called a Rota-Baxter operator of weight $\lambda$ on $\mathcal{A}$, that is, $R$ satisfies  the identity \eqref{R-B assoc}.
We denote by a triple $(\mathcal{A},\mu,R)$ the Rota-Baxter associative superalgebra.\\

%Remark 2.8. Under our assumption that k is a field, the non-zero weight can be normalized
%to weight 1. In fact, for a non-zero weight ? ? k, if ? is an O-operator of weight ? associated
%to an A-bimodule k-algebra (R, ?, ?, r), then ? is an O-operator of weight 1 associated to
%(R, ??, ?, r) and 1
%?? is an O-operator of weight 1 associated to (R, ?, ?, r).
%Note that, an A-bimodule (V, ?, r) becomes an A-bimodule k-algebra when V is equipped
%with the zero multiplication. Then a linear map ? : V ? A is an O-operator (of any weight
%?) if
%?(u) á ?(v) = ?(?(?(u))v) + ?(ur(?(v))), ?u, v ? V.
%Such a structure appeared independently in [26] under the name of generalized Rota-Baxter
%operator. Since the weight ? makes no difference in the definition, we just call V an Ooperator.
%This definition recovers the definition in Eq. (6).
%Obviously, an O-operator associated to (A, L,R) is just a Rota-Baxter operator on A.
%An O-operator can be viewed as the relative version of a Rota-Baxter operator in the sense
%that the domain and range of an O-operator might be different.

We define now Rota-Baxter operators on $\mathcal{A}$-bimodules.
\begin{df}Let $(\mathcal{A},\mu,R)$ be a Rota-Baxter associative superalgebra of weight zero. A Rota-Baxter operator on an  $\mathcal{A}$-bimodule $V$ $($relative to $R)$ is a map $R_V:V\longrightarrow V$ such that for all
 $x \in \mathcal{H}(\mathcal{A})$ and $v \in \mathcal{H}(V)$
\begin{eqnarray*}
    & & \ \ \ R(x)R_V(v) = R_V \Big(R(x)v+x R_V(v)\Big),\\
    & & \ \ \ R_V(v)R(x) = R_V \Big(R_V(v)x+v R(x) \Big).
    \end{eqnarray*}
\end{df}
We have similar definitions on Lie superalgebras.
 \begin{df}\ \begin{enumerate}
\item A Lie superalgebra is a pair $(\mathcal{A}, [~~,~~])$ consisting of a $\mathbb{Z}_2$-graded vector space $\mathcal{A}$, and an even
bilinear map $[~~,~~] : \mathcal{A}\otimes \mathcal{A} \longrightarrow \mathcal{A},~~([\mathcal{A}_i,\mathcal{A}_j]\subseteq \mathcal{A}_{i+j},~~\forall~~i,j\in \mathbb{Z}_2)$  satisfying for all $\ x,y,z \in \mathcal{H}(\mathcal{A})$, 
\begin{eqnarray}
& [x,y] = -(-1)^{|x||y|}[y,x],\quad \text{(super skew-symmetry)}\\& 
\label{H-sJ}
  [x,[y,z]] = [[x,y],z]+(-1)^{|x||y|}[y,[x,z]], \quad \text{(super-Jacobi identity)}.
\end{eqnarray}
\item Let $(\mathcal{A},[~~,~~])$ be a Lie superalgebra, $V$ be a $\mathbb{Z}_2$-graded vector space and  $\rho:\mathcal{A}\longrightarrow End(V)$ be an even linear map. The pair $(V,\rho)$ is said to be an $\mathcal{A}$-module or a representation of $(\mathcal{A},[~~,~~])$ if for all  $x,y\in \mathcal{H}(\mathcal{A})$ and $v\in \mathcal{H}(V),$
 \begin{eqnarray}\label{1-rep-Lie}
 \rho([x,y])(v) &=& \rho(x)\rho(y)v-(-1)^{|x||y|}\rho(y)\rho(x)v.
 \end{eqnarray}
 \\ The triple $(V,[ \ ,\ ]_V,\rho)$, where $[ \ ,\ ]_V$ is a super skew-symmetric bracket,  is said to be an $\mathcal{A}$-module $\K$-superalgebra if,  for $x\in \mathcal{H}(\mathcal{A})$ and $v,w\in \mathcal{H}(V),$
 $$ \rho(x)[ v,w ]_V=[ \rho(x)(v),w]_V+ (-1)^{|v||w|}[v, \rho(x)(w)]_V.$$ 
\item Let $(\mathcal{A},[~~,~~])$ be a Lie superalgebra and $(V,\rho)$ be a representation of $\mathcal{A}$. An even linear map
$T:V\longrightarrow \mathcal{A}$ is called a super $\mathcal{O}$-operator of weight $\lambda \in \mathbb{K}$ associated to $\mathcal{A}$-module $\K$-superalgebra $(V,[ \ ,\ ]_V,\rho)$ if $T$ satisfies
\begin{eqnarray*}
  [T(u),T(v)] &=& T\Big(\rho(T(u))v-(-1)^{|u||v|}\rho(T(v))u+\lambda [u,v]_V\Big),~~\forall~~u,v\in \mathcal{H}(V).
\end{eqnarray*}
\end{enumerate}
 In particular, a super $\mathcal{O}$-operator of weight $\lambda \in \mathbb{K}$ associated to the bimodule $(\mathcal{A},L_\circ,R_\circ)$ is called a Rota-Baxter operator of weight $\lambda \in \mathbb{K}$ on $(\mathcal{A},[~~,~~])$, that is, $R$ satisfies for all $x,y,z$ in $\mathcal{H}(\mathcal{A})$
\begin{equation}\label{Rota-baxter-Lie}
    [R(x),R(y)]=R\Big([R(x),y]-(-1)^{|x||y|}[R(y),x]+\lambda [x,y]\Big).
\end{equation}

The triple $(\mathcal{A},[~~,~~],R)$ refers to a Rota-Baxter Lie superalgebra \cite{WangHouBai-OperatorLieSuper}.
\end{df}
\begin{df}
Let $(\mathcal{A},[ ~~,~~ ],R)$ and $(\mathcal{A}',[ ~~,~~ ]',R')$ be two Rota-Baxter Lie superalgebras. An even homomorphism $f:(\mathcal{A},[ ~~,~~ ],R)\longrightarrow (\mathcal{A}',[ ~~,~~ ]',R')$ is said to be a morphism of two Rota-Baxter Lie  superalgebras if, for all $x,y\in \mathcal{H}(\mathcal{A}$, 
$
  f([x,y]) = [f(x),f(y)]'$
  and 
 $  f\circ R = R'\circ f.
$
\end{df}
\begin{prop}
 Let $(\mathcal{A},\mu,R)$ be a Rota-Baxter associative superalgebra of weight $\lambda \in \mathbb{K}$. Then the triple  $(\mathcal{A},[ ~~,~~ ],R)$, where $[x,y]=xy-(-1)^{|x||y|}yx$, is a Rota-Baxter Lie superalgebra of weight $\lambda \in \mathbb{K}$.
\end{prop}
\begin{prop}Let $(\mathcal{A},\mu,R)$ be a Rota-Baxter associative superalgebra of weight $\lambda=-1$. Then the binary operation defined, for any homogeneous elements $x,y$ in $\mathcal{A}$,  by
\begin{eqnarray*}
  [x,y] = R(x)y-(-1)^{|x||y|}y R(x)-x y 
   + x R(y)-(-1)^{|x||y|}R(y) x+(-1)^{|x||y|}y x,
\end{eqnarray*}
defines a Rota-Baxter Lie superalgebra $(\mathcal{A},[ ~~,~~ ],R)$ of weight $\lambda=-1$.
\end{prop}
\begin{proof}Let $x,y$ and $z$  in $\mathcal{H}(\mathcal{A})$. We show that the super-Jacobi identity is satisfied
\begin{align*}
&  \ \ \ (-1)^{|x||z|}[x,[y,z]]+(-1)^{|y||z|}[z,[x,y]]+(-1)^{|x||y|}[y,[z,x]] \\
&  \ \ \ =\circlearrowleft_{x,y,z}\Big((-1)^{|x||z|}R(x)(R(y)z)-(-1)^{|x||y|}(R(y)z)R(x)-(-1)^{|x||z|}x(R(y)z)\\
&  \ \ \ \hskip0.5cm -(-1)^{|z|(|x|+|y|)}R(x)(z R(y))+ (-1)^{|y|(|x|+|z|)} (z R(y))R(x)-(-1)^{|z|(|x|+|y|)}x(z R(y))\\
&  \ \ \ \hskip0.5cm -(-1)^{|z||x|}R(x)(y z)+(-1)^{|x||y|}(y z)R(x)+(-1)^{|z||x|}x(y z)-(-1)^{|x||y|}R(R(y)z)x \\
&  \ \ \ \hskip0.5cm +(-1)^{|z||x|}x R(R(y)z)+(-1)^{|x||y|}(R(y)z)x+(-1)^{|y|(|x|+|z|)}R(z R(y))x-(-1)^{|z|(|x|+|y|)}x R(z R(y))\\
&  \ \ \ \hskip0.5cm -(-1)^{|y|(|x|+|z|)}(z R(y))x+(-1)^{|x||y|} R(y z)x-(-1)^{|z||x|}x R(y z)-(-1)^{|z|(|x|+|y|)}R(x)(R(z)y)\\
&  \ \ \ \hskip0.5cm +(-1)^{|y|(|x|+|z|)}(R(z)y)R(x)+(-1)^{|z|(|x|+|y|)}x(R(z)y)+(-1)^{|x||z|}R(x)(y R(z))-(-1)^{|x||y|}(y R(z))R(x)\\
&  \ \ \ \hskip0.5cm -(-1)^{|z||x|}x(y R(z))+(-1)^{|z|(|x|+|y|)}R(x)(z y)-(-1)^{|y|(|x|+|z|)}(z y)R(x)-(-1)^{|z|(|x|+|y|)}x(z y)\\
&  \ \ \ \hskip0.5cm +(-1)^{|y|(|x|+|z|)}R(R(z)y)x-(-1)^{|z|(|x|+|y|)}x R(R(z)y)+(-1)^{|y|(|x|+|z|)}R(R(z)y)x\\
&  \ \ \ \hskip0.5cm -(-1)^{|z|(|x|+|y|)}x R(R(z)y)
 -(-1)^{|y|(|x|+|z|)}(R(z)y)x-(-1)^{|x||y|}R(y R(z))x+(-1)^{|x||z|}x R(y R(z))\\
&  \ \ \ \hskip0.5cm +(-1)^{|x||y|}(y R(z))x-(-1)^{|y|(|x|+|z|)}R(z y)x+(-1)^{|z|(|x|+|y|)}x R(z y)+(-1)^{|y|(|x|+|z|)}(z y)x\Big).
\end{align*}
The above sum vanishes by associativity and the Rota-Baxter super-identity $(\ref{R-B assoc})$ with $\lambda=-1$.\\
It easy to show that $R$ is a Rota-Baxter operator of weight $\lambda=-1$ on the Lie superalgebra $(\mathcal{A},[ ~~,~~ ])$. Therefore $(\mathcal{A},[ ~~,~~ ],R)$ is a Rota-Baxter Lie superalgebra of weight $\lambda=-1$. Which ends the proof.
\end{proof}
We introduce the notion of super $\mathcal{O}$-operators of pre-Lie superalgebras. Then we study the relations among Lie superalgebras and pre-Lie superalgebras.
\begin{df}Let $\mathcal{A}$ be a $\mathbb{Z}_2$-graded vector space an the even binary operation denoted by
$\circ:\mathcal{A}\otimes \mathcal{A}\longrightarrow \mathcal{A}$. The pair $(\mathcal{A},\circ)$ is called a pre-Lie superalgebra if, for  $x,y,z$ in $\mathcal{H}(\mathcal{A})$, the associator
$$as(x,y,z)=(x\circ y)\circ z-x\circ(y\circ z)$$
is super-symmetric  in $x$ and $y$, that is,
$as(x,y,z)=(-1)^{|x||y|}as(y,x,z)$, or equivalently \begin{eqnarray}\label{egali-hom-left}
(x \circ y)\circ z-x \circ(y \circ z)&=&(-1)^{|x||y|}\Big((y \circ x)\circ z-y\circ(x \circ z)\Big).
 \end{eqnarray}
The identity (\ref{egali-hom-left}) is called pre-Lie super-identity.
\end{df}
\begin{df}\label{defi-oper-pre}\
Let $(\mathcal{A},\circ)$ be a pre-Lie superalgebra. \begin{enumerate}
\item  Let $V$ be a $\mathbb{Z}_2$-graded vector space and $l,r:\mathcal{A}\longrightarrow End(V)$ be two even linear maps. The triple $(V,l,r)$ is said to be an $\mathcal{A}$-bimodule of $(\mathcal{A},\circ)$ if,  for   $x,y \in \mathcal{H}(\mathcal{A})$ and $v\in \mathcal{H}(V),$
% \begin{eqnarray}\label{1-rep-pre}
% & l(x)\circ l(y)v-l(x\circ y)v =(-1)^{|x||y|}\Big(l(y)\circ l(x)v-l(y\circ x) v \Big),\\ &
% l(x)\circ r(y)v-(-1)^{|x||y|}r(y)\circ l(x)v = r(x\circ y) v-(-1)^{|x||y|}r(y)\circ r(x)v
% \end{eqnarray}
 \begin{eqnarray}\label{1-rep-pre}
 & l(x) l(y)v-l(x\circ y)v =(-1)^{|x||y|}\Big(l(y) l(x)v-l(y\circ x) v \Big),\\ &
 l(x) r(y)v-r(y) l(x)v =(-1)^{|x||v|}( r(x\circ y) v-r(y) r(x)v).
 \end{eqnarray}
 Moreover, the quadruple $(V,\circ_V,l,r)$ is said to be an  $\mathcal{A}$-bimodule $\K$-superalgebra if $(V,l,r)$  is an  $\mathcal{A}$-bimodule compatible with the multiplication $\circ_V$ on $V$, that is, for $x,y \in \mathcal{H}(\mathcal{A})$ and $v,w\in \mathcal{H}(V),$ 
\begin{eqnarray*}&  l(x)(v\circ_V w)-l(x)(v)\circ_V w= (-1)^{|x||v|} (v\circ_V l(x)(w))-r(x)(v)\circ_V w),\\ & ~~ r(x)(v\circ_V w)- v\circ_V r(x)(w)=  (-1)^{|v||w|}(r(x)(w\circ_V v)- w\circ_V r(x)(v)).
\end{eqnarray*}
\item Let  $(V,\circ_V,l,r)$ be an $\mathcal{A}$-bimodule $\K$-superalgebra. An even linear map
$T:V\longrightarrow \mathcal{A}$ is called a super $\mathcal{O}$-operator of weight $\lambda \in \mathbb{K}$ associated to $(V,\circ_V,l,r)$ if it  satisfies:
\begin{equation}\label{opera2}
   T(u)\circ T(v)= T\Big(l(T(u))v+(-1)^{|u||v|}r(T(v))u +\lambda u\circ_V v\Big),~~\forall~~~u,v\in \mathcal{H}(V).
\end{equation}
\end{enumerate}
 In particular, a super $\mathcal{O}$-operator of weight $\lambda \in \mathbb{K}$ associated to the  $\mathcal{A}$-bimodule $(\mathcal{A},L_\circ,R_\circ)$ is called a Rota-Baxter operator of weight $\lambda$ on $(\mathcal{A},\circ)$, that is, $R$ satisfies
\begin{equation}\label{Rota-baxter-Lie}
    R(x)\circ R(y)=R\Big(R(x)\circ y+x\circ R(y)+\lambda x\circ y\Big)
\end{equation}
for all $x,y,z$ in $\mathcal{H}(\mathcal{A})$.
\end{df}
\begin{prop}
Let $(\mathcal{A},\circ)$ be a pre-Lie superalgebra.
\begin{enumerate}
\item The commutator
\begin{eqnarray*}
    [x,y]&=& x\circ y-(-1)^{|x||y|} y\circ x
\end{eqnarray*}
defines a Lie superalgebra $(\mathcal{A},[~~,~~])$ which is called the sub-adjacent Lie superalgebra of $\mathcal{A}$ and $\mathcal{A}$ is also called a compatible pre-Lie superalgebra structure on the Lie superalgebra.
\item The map  $L_\circ$ gives a representation of the Lie superalgebra $(\mathcal{A},[~~,~~])$, that is,
$$L_\circ([x,y])=L_\circ(x) L_\circ(y)-(-1)^{|x||y|}L_\circ(y) L_\circ(x).$$
\end{enumerate}
\end{prop}
\begin{cor}Let $(\mathcal{A},\circ)$ be a pre-Lie superalgebra and $(V,l,r)$ be an  $\mathcal{A}$-bimodule. Let $(\mathcal{A},[~~,~~])$ be the
subadjacent  Lie superalgebra. If $T$ is a super $\mathcal{O}$-operator associated to $(V,l,r)$, then $T$ is a super $\mathcal{O}$-operator of $(\mathcal{A},[~~,~~])$ associated to $(V,l-r,r-l)$.
\end{cor}
\begin{thm}Let $\mathcal{A}_{1}=(\mathcal{A},\circ,R)$ be  a Rota-Baxter pre-Lie superalgebra of weight zero. Then $\mathcal{A}_{2}=(\mathcal{A},\ast,R)$ is a Rota-Baxter pre-Lie superalgebra of weight zero, where the even binary operation $"\ast"$ is defined by
\begin{eqnarray*}
    x \ast y&=& R(x)\circ y-(-1)^{|x||y|}y \circ R(x).
\end{eqnarray*}
\end{thm}
\begin{proof}Let $x,y$ and $z$ be a homogeneous elements in $\mathcal{A}$. Then we have
\begin{eqnarray*}
x\ast (y\ast z)&=& R(x)\circ(R(y)\circ z)-(-1)^{|x|(|y|+|z|)}(R(y)\circ z)\circ R(x)\\
&-&(-1)^{|y||z|}R(x)\circ(z\circ R(y))+(-1)^{|x|(|y|+|z|)}(z\circ R(y))\circ R(x),
\end{eqnarray*}
and
\begin{eqnarray*}
(x \ast y)\ast z&=& R (R(x)\circ y)\circ z-(-1)^{|z|(|x|+|y|)}z \circ R(R(x)\circ y)\\
&-&(-1)^{|x||y|}R(y \circ R(x))\circ z+(-1)^{|z|(|x|+|y|)+|x||y|}z\circ R(y \circ R(x)).
\end{eqnarray*}%
Subtracting the above terms, switching $x$ and $y$, and then subtracting the result yield:
\begin{small}
\begin{eqnarray*}
  && \ \    as_{\mathcal{A}_{2}}(x,y,z)-(-1)^{|x||y|}as_{\mathcal{A}_{2}}(y,x,z) \\
  && \ \ = (-1)^{|x||y|}\Big(R(y\circ R(x))\circ z+R(R(y)\circ x)\circ z \Big) -\Big(R(R(x)\circ y)\circ z+R(x\circ R(y))\circ z\Big)\\
  && \ \  + (-1)^{|z|(|x|+|y|)}\Big(z\circ R(R(x)\circ y)+z\circ R(x\circ R(y))\Big)
   - (-1)^{|z|(|x|+|y|)+|x||y|} \Big(z\circ R(R(x)\circ y)+z\circ R(R(y)\circ x) \Big) \\
  && \ \  +  R(x)\circ (R(y)\circ z)- (-1)^{|x|(|y|+|z|)}(R(y)\circ z)\circ R(x)- (-1)^{|y||z|}R(x)\circ (z \circ R(y))\\
&& \ \   + (-1)^{|x|(|y|+|z|)+|y||z|}(z \circ R(y))\circ R(x)
   - (-1)^{|x||y|}R(y)\circ (R(x)\circ z) - (-1)^{|y||z|}(R(x)\circ z)\circ R(y)\\
  && \ \  -(-1)^{|x|(|y|+|z|)}R(y)\circ (z\circ R(x))+(-1)^{|z|(|x|+|y|)} (z \circ R(x))\circ R(y)\\
  && \ \  = R(x)\circ (R(y)\circ z)-(R(x)\circ R(y))\circ z-(-1)^{|x||y|}R(y)\circ (R(x)\circ z)+(-1)^{|x||y|}(R(y)\circ R(x))\circ z\\
  && \ \  +(-1)^{|z|(|x|+|y|)}\Big(z\circ (R(x)\circ R(y))-(z\circ R(x))\circ R(y)\Big)
   -(-1)^{|y||z|}\Big( R(x)\circ(z\circ R(y))-(R(x)\circ z)\circ R(y)\Big)\\
  && \ \  + (-1)^{|x|(|y|+|z|)}\Big( R(y)\circ (z\circ R(x)) -(R(y)\circ z)\circ R(x)\Big)
  \\ && \ \ 
   -(-1)^{|z|(|x|+|y|)+|x||y|}\Big( z\circ (R(y)\circ R(x))-(z \circ R(y))\circ R(x)\Big)\\
  && \ \ \ = \Big(as_{\mathcal{A}_{1}}(x,y,z)-(-1)^{|x||y|}as_{\mathcal{A}_{1}}(y,x,z)\Big)
  +(-1)^{|z|(|x|+|y|)}\Big(as_{\mathcal{A}_{1}}(z,x,y)-(-1)^{|x||z|}as_{\mathcal{A}_{1}}(x,z,y) \Big)\\
  && \ \  +(-1)^{|x|(|y|+|z|)}\Big(as_{\mathcal{A}_{1}}(y,z,x)-(-1)^{|y||z|}as_{\mathcal{A}_{1}}(z,y,x) \Big)\\
   && \ \  =0.
  \end{eqnarray*}
  \end{small}
  Then $\mathcal{A}_{2}$ is a pre-Lie superalgebra, which ends the proof.
\end{proof}
Now, we construct pre-Lie superalgebras using super $\mathcal{O}$-operators on Lie superalgebras.
\begin{prop}Let $(\mathcal{A},[~~,~~])$ be a Lie superalgebra and $(V,\rho)$ be a representation of $\mathcal{A}$. Suppose that $T:V\longrightarrow \mathcal{A}$ is a super $\mathcal{O}$-operator of weight zero associated to $(V,\rho)$. Then, the  even bilinear map
$$u\circ v=\rho(T(u))v,~~\forall~~u,v\in \mathcal{H}(V)$$
defines a pre-Lie superalgebra structure on $\mathcal{A}$.
\end{prop}
\begin{proof}Let $u,v$ and $w$ be  in $\mathcal{H}(V)$. We have
\begin{eqnarray*}
 &  (u\circ v)\circ w-u\circ(v\circ w) = \rho(T(\rho(T(u))v))w-\rho(T(u))\rho(T(v))w,\\
 &  (-1)^{|u||v|}\Big((v\circ u)\circ w-v\circ(u\circ w)\Big) = (-1)^{|u||v|}\Big( \rho(T(\rho(T(v))u))w-\rho(T(v))\rho(T(u))w\Big).
\end{eqnarray*}
Hence \begin{eqnarray*}
      & & \ \ \ (u\circ v)\circ w-u\circ(v\circ w)-(-1)^{|u||v|}\Big((v\circ u)\circ w-v\circ(u\circ w)\Big)\\
      & & \ \ \  =\rho T\Big(\rho(T(u))v-(-1)^{|u||v|}\rho(T(v))u\Big)w-\rho(T(u))\rho(T(v))w+(-1)^{|u||v|}\rho(T(v))\rho(T(u))w\\
      & & \ \ \  =\rho([T(u),T(v)])-\rho(T(u))\rho(T(v))w+(-1)^{|u||v|}\rho(T(v))\rho(T(u))w\\
      & & \ \ \  =0.
      \end{eqnarray*}
      Therefore $(\mathcal{A},\circ)$ is a pre-Lie superalgebra.
\end{proof}
\begin{rem}\label{rota-baxter=prelie}Let $(\mathcal{A},[~~,~~])$ be a Lie superalgebra and $R$ be the super $\mathcal{O}$-operator $($of weight zero$)$ associated to the adjoint representation $(\mathcal{A},ad)$. Then the even binary operation given by
$x\circ y=[R(x),y]$, for all $x, y \in \mathcal{H}(\mathcal{A})$, 
defines a pre-Lie superalgebra structure on $\mathcal{A}$.
\end{rem}
\begin{exa}\label{osp1,2}In this example, we calculate  Rota-Baxter operators of weight zero on the Lie superalgebra $\mathfrak{osp(1,2)}$ and we give the corresponding pre-Lie superalgebras. Starting from the orthosymplectic Lie superalgebra, we consider in the sequal the matrix realization of this superalgebra.\\ Let $\mathfrak{osp(1,2)}=\mathcal{A}_0\oplus \mathcal{A}_1$ be the Lie superalgebra where $\mathcal{A}_0$ is spanned by
$$ e_1=\left(
     \begin{array}{ccc}
       1 & 0 & 0 \\
       0 & 0 & 0 \\
       0 & 0 & -1 \\
     \end{array}
   \right),~~
e_2=\left(
     \begin{array}{ccc}
       0 & 0 & 1 \\
       0 & 0 & 0 \\
       0 & 0 & 0 \\
     \end{array}
   \right),~~e_3=\left(
     \begin{array}{ccc}
       0 & 0 & 0 \\
       0 & 0 & 0 \\
       1 & 0 & 0 \\
     \end{array}
   \right),$$
and $\mathcal{A}_1$ is spanned by
$$ e_4=\left(
     \begin{array}{ccc}
       0 & 0 & 0 \\
       1 & 0 & 0 \\
       0 & 1 & 0 \\
     \end{array}
   \right),~~
e_5=\left(
     \begin{array}{ccc}
       0 & 1 & 0 \\
       0 & 0 & -1 \\
       0 & 0 & 0 \\
     \end{array}
   \right).$$
The defining relations $($we give only the ones with non-zero values in the right-hand side$)$ are
\begin{eqnarray*}
  && \ \ \ [e_1,e_2] = 2e_2,~~[e_1,e_3]=-2e_3,~~[e_2,e_3]=e_1,\\
  && \ \ \  [e_3,e_5] = e_4,\hskip0.3cm [e_2,e_4]=e_5,\hskip0.6cm [e_1,e_4]=-e_4,~~[e_1,e_5]=e_5, \\
  && \ \ \  [e_5,e_4]=e_1,\hskip0.3cm [e_5,e_5]=-2e_2,\hskip0.15cm[e_4,e_4]=2e_3.
\end{eqnarray*}
The Rota-Baxter operators of weight zero on the Lie superalgebra $\mathfrak{osp(1,2)}$ with respect to the homogeneous basis $\{e_1,e_2,e_3,e_4,e_5\}$ are:
\begin{itemize}
\item[] \hskip-1.3cm $R_{1}(e_1)=a_1 e_1+a_2e_2-\frac{8a_1^2a_3}{(2a_3+a_2)^2}e_3,\ R_{1}(e_2)=-\frac{2a_2a_1^2}{(2a_3+a_2)^2}e_1+\frac{(2a_3-3a_2)a_1}{2(2a_3+a_2)}e_2+\frac{2a_1^3}{(2a_3+a_2)^2}e_3,\ R_{1}(e_3)=a_3e_1+\frac{(2a_3+a_2)^2}{8a_1}e_2+\frac{a_1(a_2-6a_3)}{2(2a_3+a_2)}e_3,\ R_{1}(e_4)=0,\ R_{1}(e_5)=0,~~a_1\neq 0,~~a_2\neq -2a_3.$
\item[]\hskip-1.3cm  $R_{2}(e_1)=a_1e_1+a_2e_2,\ R_{2}(e_2)=-\frac{2a_1^2}{a_2}e_1-\frac{3a_1}{2}e_2+\frac{2a_1^3}{a_2^2}e_3,\ R_{2}(e_3)=\frac{a_2^2}{8a_1}e_2+\frac{a_1}{2}e_3,\ R_{2}(e_4)=0,\ R_{2}(e_5)=0,~~a_1\neq 0,~~a_2\neq 0.$
\item[]\hskip-1.3cm  $R_{3}(e_1)=a_1e_1-\frac{2a_1^2}{a_3}e_3,\ R_{3}(e_2)=\frac{a_1}{2}e_2+\frac{a_1^3}{2a_3^2}e_3,\ R_{3}(e_3)=a_3e_1+\frac{a_3^2}{2a_1}e_2-\frac{3a_1}{2}e_3,\ R_{3}(e_4)=0,\ R_{3}(e_5)=0,~~a_1\neq 0,~~a_3\neq 0.$
\item[]\hskip-1.3cm  $R_{4}(e_1)=0,\ R_{4}(e_2)=0,\ R_{4}(e_3)=a_3e_1+a_4e_2,\ R_{4}(e_4)=0,\ R_{4}(e_5)=0.$
\item[]\hskip-1.3cm  $R_{5}(e_1)=a_1e_1-4a_3e_2-\frac{2a_1^2}{a_3}e_3,\ R_{5}(e_2)=-\frac{a_1^2}{4a_3}e_1+a_1e_2+\frac{a_1^3}{2a_3^2}e_3,\ R_{5}(e_3)=a_3e_1-\frac{4a_3^2}{a_1}e_2-2a_1e_3,\ R_{5}(e_4)=0,\ R_{5}(e_5)=0,~~a_1\neq 0,~~a_3\neq 0.$
\item[]\hskip-1.3cm  $R_{6}(e_1)=0,\ R_{6}(e_2)=0,\ R_{6}(e_3)=a_3e_1,\ R_{6}(e_4)=0,\ R_{6}(e_5)=0.$
\item[]\hskip-1.3cm  $R_{7}(e_1)=a_1e_1+\frac{2a_1a_5}{a_6}e_2+a_6e_3,\ R_{7}(e_2)=\frac{a_6}{2}e_1+a_5e_2+\frac{a_6^2}{2a_1}e_3,\ R_{7}(e_3)=\frac{a_1a_5}{a_6}e_1+\frac{2a_1a_5^2}{a_6^2}e_2+a_5e_3,\ R_{7}(e_4)=0,\ R_{7}(e_5)=0,~~a_1\neq 0,~~a_6\neq 0.$
\item[]\hskip-1.3cm  $R_{8}(e_1)=a_1e_1+a_6e_3,\ R_{8}(e_2)=\frac{a_6}{2}e_1+\frac{a_6^2}{2a_1}e_3,\ R_{8}(e_3)=0,\ R_{8}(e_4)=0,\ R_{8}(e_5)=0,~~a_1\neq 0.$
\item[]\hskip-1.3cm  $R_{9}(e_1)=a_1e_1+a_2e_2,\ R_{9}(e_2)=0,\ R_{9}(e_3)=\frac{a_2}{2}e_1+\frac{a_2^2}{2a_1}e_2,\ R_{9}(e_4)=0,\ R_{9}(e_5)=0,~~a_1\neq 0.$
\item[]\hskip-1.3cm  $R_{10}(e_1)=0,\ R_{10}(e_2)=a_7e_1+a_8e_3,\ R_{10}(e_3)=0,\ R_{10}(e_4)=0,\ R_{10}(e_5)=0.$
\item[]\hskip-1.3cm $R_{11}(e_1)=0,\ R_{11}(e_2)=a_5e_2+a_8e_3,\ R_{11}(e_3)=\frac{a_5^2}{a_8}e_2+a_5e_3,\ R_{11}(e_4)=0,\ R_{11}(e_5)=0,~~a_8\neq  0.$
\item[]\hskip-1.3cm  $R_{12}(e_1)=a_6e_3,\ R_{12}(e_2)=-\frac{a_6}{2}e_1+a_8e_3,\ R_{12}(e_3)=0,\ R_{12}(e_4)=0,\ R_{12}(e_5)=0.$
\item[]\hskip-1.3cm  $R_{13}(e_1)=0,\ R_{13}(e_2)=a_7e_1,\ R_{13}(e_3)=0,\ R_{13}(e_4)=0,\ R_{13}(e_5)=0.$
\item[]\hskip-1.3cm  $R_{14}(e_1)=0,\ R_{14}(e_2)=0,\ R_{14}(e_3)=a_4e_2,\ R_{14}(e_4)=0,\ R_{14}(e_5)=0.$
\item[]\hskip-1.3cm  $R_{15}(e_1)=0,\ R_{15}(e_2)=0,\ R_{15}(e_3)=0,\ R_{15}(e_4)=0,\ R_{15}(e_5)=0.$
\item[]\hskip-1.3cm  $R_{16}(e_1)=a_1 e_1+a_2e_2+\frac{a_1^2(a_2-4a_3)}{4a_3^2}e_3,\ R_{16}(e_2)=-\frac{a_1^2}{4a_3}e_1-\frac{a_1a_2}{4a_3}e_2+\frac{a_1^3(4a_3-a_2)}{16a_3^3}e_3,\ R_{16}(e_3)=a_3e_1+\frac{a_2a_3}{a_1}e_2+\frac{a_1(a_2-4a_3)}{4a_3}e_3,\ R_{16}(e_4)=0,\ R_{16}(e_5)=0,~~a_1\neq 0,~~a_3\neq 0.$
\item[]\hskip-1.3cm  $R_{17}(e_1)=a_1e_1-\frac{a_1^2}{a_3}e_3,\ R_{17}(e_2)=-\frac{a_1^2}{4a_3}e_1+\frac{a_1^3}{4a_3^2}e_3,\ R_{17}(e_3)=a_3e_1-a_1e_3,\ R_{17}(e_4)=0,\ R_{17}(e_5)=0,~~a_3\neq 0.$
\item[]\hskip-1.3cm  $R_{18}(e_1)=a_1e_1-4a_3e_2-\frac{2a_1^2}{a_3}e_3,\ R_{18}(e_2)=\frac{2a_1^2}{a_3}e_1-\frac{7a_1}{2}e_2+\frac{a_1^3}{2a_3^2}e_3,\ R_{18}(e_3)=a_3e_1+\frac{a_3^2}{2a_1}e_2+\frac{5a_1}{2}e_3,\ R_{18}(e_4)=0,\ R_{18}(e_5)=0,~~a_1\neq 0,~~a_3\neq 0.$
\item[]\hskip-1.3cm  $R_{19}(e_1)=a_1e_1+4a_3e_2,\ R_{4}(19)=-\frac{a_1^2}{4a_3}e_1-a_1e_2,\ R_{19}(e_3)=a_3e_1+\frac{4a_3^2}{a_1}e_2,\ R_{19}(e_4)=0,\ R_{19}(e_5)=0,~~a_1\neq 0,~~a_3\neq 0$
\item[]\hskip-1.3cm  $R_{20}(e_1)=0,\ R_{20}(e_2)=0,\ R_{20}(e_3)=a_3e_1+a_4e_2,\ R_{20}(e_4)=0,\ R_{20}(e_5)=0.$
\item[]\hskip-1.3cm  $R_{21}(e_1)=-2a_3e_2,\ R_{21}(e_2)=0,\ R_{21}(e_3)=a_3e_1+a_4e_2,\ R_{21}(e_4)=0,\ R_{21}(e_5)=0.$
\item[]\hskip-1.3cm  $R_{22}(e_1)=\frac{4a_5^2}{a_6}e_2+a_6e_3,\ R_{22}(e_2)=-a_5e_2-\frac{a_6^2}{4a_5}e_3,\ R_{22}(e_3)=\frac{4a_5^3}{a_6^2}e_2+a_5e_3,\ R_{23}(e_4)=0,\ R_{23}(e_5)=0,~~a_5\neq 0,~~a_6\neq 0.$
\item[]\hskip-1.3cm  $R_{23}(e_1)=a_2e_1,\ R_{23}(e_2)=0,\ R_{23}(e_3)=-\frac{a_2}{2}e_2+a_4e_2,\ R_{23}(e_4)=0,\ R_{23}(e_5)=0.$
\item[]\hskip-1.3cm  $R_{24}(e_1)=0,\ R_{24}(e_2)=a_6e_1+a_7e_3,\ R_{24}(e_3)=0,\ R_{24}(e_4)=0,\ R_{24}(e_5)=0.$
\item[]\hskip-1.3cm  $R_{25}(e_1)=0\ R_{25}(e_2)=0,\ R_{25}(e_3)=a_4e_2,\ R_{25}(e_4)=0,\ R_{25}(e_5)=0.$
\item[]\hskip-1.3cm  $R_{26}(e_1)=0,\ R_{26}(e_2)=a_5e_2+\frac{a_5^2}{a_4}e_3,\ R_{26}(e_3)=a_4e_2+a_5e_3,\ R_{26}(e_4)=0,\ R_{26}(e_5)=0,~~a_4\neq 0.$
\item[]\hskip-1.3cm  $R_{27}(e_1)=a_1e_1,\ R_{27}(e_2)=0,\ R_{27}(e_3)=0,\ R_{27}(e_4)=0,\ R_{27}(e_5)=0.$
\item[]\hskip-1.3cm  $R_{28}(e_1)=0,\ R_{28}(e_2)=a_7e_3,\ R_{28}(e_3)=0,\ R_{28}(e_4)=0,\ R_{28}(e_5)=0.$
\item[]\hskip-1.3cm  $R_{29}(e_1)=a_9e_1-\frac{2a_9^2}{a_{10}}e_3,\ R_{29}(e_2)=\frac{a_9}{2}e_2+\frac{a_9^3}{2a_{10}^2}e_3,\ R_{29}(e_3)=a_{10}e_1+\frac{a_5^2}{2a_9}e_2-\frac{3a_9}{2}e_3,\ R_{29}(e_4)=-a_9e_4+a_{10}e_5,\ R_{29}(e_5)=-\frac{a_9^2}{a_{10}}e_4+a_9e_5,~~a_9\neq 0,~~a_{10}\neq 0.$
\item[]\hskip-1.3cm  $R_{30}(e_1)=-a_{10}e_2,\ R_{30}(e_2)=0,\ R_{30}(e_3)=\frac{a_{10}}{2}e_1+a_4e_2,\ R_{30}(e_4)=a_{10}e_5,\ R_{30}(e_5)=0.$
\item[]\hskip-1.3cm  $R_{31}(e_1)=a_{11}e_3,\ R_{31}(e_2)=-\frac{a_{11}}{2}e_1+a_8e_3,\ R_{31}(e_3)=0,\ R_{31}(e_4)=0,\ R_{31}(e_5)=a_{11}e_4.$
\end{itemize}
The constants  $a_i$ are parameters.
\end{exa}
\begin{df}Let $(\mathcal{A},[~~,~~],R)$ be a Rota-Baxter Lie superalgebra  of weight zero. A Rota-Baxter operator on an $\mathcal{A}$-module $V$ $($relative to $R)$ is a map $R_V:V\longrightarrow V$ such that, for all $x \in \mathcal{H}(\mathcal{A})$ and $v \in \mathcal{H}(V)$,
\begin{eqnarray*}
    & &  [R(x),R_V(v)] = R_V \Big([R(x),v]+[x,R_V(v)]\Big),\\
    & & [R_V(v),R(x)] = R_V \Big([R_V(v),x]+[v,R(x)]\Big),
    \end{eqnarray*}
    where the action $\rho(x)(v)$ is denoted by $[x,v]$.
\end{df}
\begin{prop}Let $(\mathcal{A},[~~,~~],R)$ be a Rota-Baxter Lie superalgebra of weight zero,
$V$ an $\mathcal{A}$-module and $R_V$ a Rota-Baxter operator on $V$. Define new actions of $\mathcal{A}$ on $V$ by
$$ x\circ v=[R(x),v],~~~~v\circ x=[R_V(v),x].$$
Equipped with these actions, $V$ is a bimodule over the pre-Lie superalgebra (Remark \ref{rota-baxter=prelie}).
\end{prop}
\begin{proof}Let $x,y$ be a homogeneous elements in $\mathcal{A}$ and $v$ in $V$.  We have
\begin{align*}
 &   l(x)r(y)(v)-(-1)^{|x||y|}r(y)l(x)-r(x\circ y)(v) +(-1)^{|x||y|}r(y)r(x)(v)\\
 &  = (-1)^{|y||v|}[R(x),[R_V(v),y]]-(-1)^{|y||v|}[R_V([R(x),v]),y]\\
 &   -(-1)^{|v|(|x|+|y|)}[R_V(v),[R(x),y]]+(-1)^{|v|(|x|+|y|)}[R_V([R_V(v),x]),y]\\
 &  = (-1)^{|y||v|}[R(x),[R_V(v),y]]-(-1)^{|y||v|}\Big([R_V([R(x),v]+[x,R_V(v)]),y]\Big)
  -(-1)^{|v|(|x|+|y|)}[R_V(v),[R(x),y]]\\
 &   =(-1)^{|y||v|}[R(x),[R_V(v),y]]-(-1)^{|y||v|}[[R(x),R_V(v)],y]-(-1)^{|v|(|x|+|y|)}[R_V(v),[R(x),y]]\\
 &  =0.
\end{align*}
Then $$l(x)r(y)(v)-(-1)^{|x||y|}r(y)l(x)= r(x\circ y)(v) -(-1)^{|x||y|}r(y)r(x)(v).$$
Similarly, we show that $l(x)\circ l(y)v-l(x\circ y)v =(-1)^{|x||y|}\Big(l(y)\circ l(x)v-l(y\circ x) v\Big).$
\end{proof}
Now, we construct a functor from a full subcategory of the category of Rota-Baxter
Lie-admissible (or associative) superalgebras to the category of pre-Lie superalgebras. The Lie-admissible algebras were studied by A. A. Albert in $1948$ and  M. Goze and E. Remm, in $2004$,  introduced the notion of $G$-associative algebras where $G$ is a subgroup of the permutation group $S_3$ $($see \cite{Goze}$)$. The graded case was studied by F. Ammar and A. Makhlouf in $2010$,  $($see \cite{Ammar-Makhlouf} for more details$)$.
\begin{df}\ \begin{enumerate}
\item A Lie-admissible superalgebra is a superalgebra $(\mathcal{A},\mu)$ in which the super-commutator
bracket, defined for all homogeneous $x,y$ in $\mathcal{A}$ by
$$[x,y] = \mu(x,y)-(-1)^{|x||y|}\mu(y,x),$$
satisfies the super-Jacobi identity (\ref{H-sJ}).
\item Let $G$ be a subgroup of the permutation group $\mathcal{S}_3$. A Rota-Baxter $G$-associative superalgebra of weight $\lambda \in \mathbb{K}$ is a $G$-associative
superalgebra $(\mathcal{A},\cdot)$ together with an even linear self-map $R: \mathcal{A}\longrightarrow \mathcal{A}$ that satisfies the identity
\begin{equation}\label{iden-Rota}
    R(x) \cdot R(y) = R(R(x)\cdot y + x \cdot R(y) -\lambda x \cdot y),
\end{equation}
for all homogeneous elements $x,y,z$ in $\mathcal{A}$.
\end{enumerate}
\end{df}
\begin{thm}
Let $(\mathcal{A},\cdot,R)$ be a Rota-Baxter Lie-admissible superalgebra of weight zero. Define the even binary operation $"\ast"$ on  any homogeneous element $x,y\in \mathcal{A}$ by
\begin{equation}\label{Rota-baxter hom}
x \ast y = R(x)\cdot y -(-1)^{|x||y|} y \cdot R(x)= [R(x), y].
\end{equation}
Then $ \mathcal{A}_{L}= (\mathcal{A},\ast)$ is a pre-Lie superalgebra.
\end{thm}
\begin{proof}For all homogeneous elements $x,y$ and $z$ in $\mathcal{A}$ we have
\begin{eqnarray*}
  & x\ast(y\ast z) = R(x)\cdot(R(y)\cdot z)-(-1)^{|y||z|}R(x)\cdot(z\cdot R(y))- (-1)^{|x|(|y|+|z|)}(R(y)\cdot z)\cdot R(x) \\
   &+(-1)^{|x|(|y|+|z|)+|y||z|}(z\cdot R(y))\cdot R(x),
  \end{eqnarray*}
and
\begin{eqnarray*}
&  (x\ast y)\ast z = R(R(x)\cdot y)\cdot z-(-1)^{|x||y|}R(y\cdot R(x))\cdot z- (-1)^{|z|(|x|+|y|)}(R(y)z\cdot R(R(x)\cdot y) \\
   &+(-1)^{|z|(|x|+|y|)+|x||y|}z\cdot R(y\cdot R(x)).
  \end{eqnarray*}
Subtracting the above terms, switching $x$ and $y$, and then subtracting the result yield:
\begin{eqnarray*}
  && as_{\mathcal{A}_{L}}(x,y,z)-(-1)^{|x||y|}as_{\mathcal{A}_{L}}(y,x,z) \\
   &=&x\ast(y \ast z)-(x\ast y)\ast z-(-1)^{|x||y|}y\ast x \ast z)+(-1)^{|x||y|}(y\ast x)\ast z  \\
  &=& \underbrace{R(x)\cdot(R(y)\cdot z)-(-1)^{|y||z|}R(x)\cdot(z\cdot R(y))}\\
   &-& \underbrace{(-1)^{|x|(|y|+|z|)}(R(y)\cdot z)\cdot R(x)
   +(-1)^{|x|(|y|+|z|)+|y||z|}(z\cdot R(y))\cdot R(x)}\\
   &-& R(R(x)\cdot y)\cdot z+(-1)^{|x||y|}R(y\cdot R(x))\cdot z \\
   &+& (-1)^{|z|(|x|+|y|)}z\cdot R(R(x)\cdot y)
   -(-1)^{|z|(|x|+|y|)+|x||y|}z\cdot R(y \cdot R(x))\\
   &-&\underbrace{(-1)^{|x||y|}R(y)\cdot(R(x)\cdot z)+(-1)^{|x|(|y|+|z|)}R(y)\cdot(z\cdot R(x))}\\
   &+&\underbrace{(-1)^{|y||z|}(R(x)\cdot z)\cdot R(y)
   -(-1)^{|z|(|x|+|y|)}(z\cdot R(x))\cdot R(y)}\\
   &+&(-1)^{|x||y|} R(R(y)\cdot x)\cdot z-R(x \cdot R(y))\cdot z \\
   &-& (-1)^{|z|(|x|+|y|)+|x||y|}z\cdot R(R(y)\cdot x)
   +(-1)^{|z|(|x|+|y|)}z\cdot R(x \cdot R(y)).
\end{eqnarray*}
Using the Rota-Baxter super-identity (\ref{iden-Rota}) with $\lambda=0$, we obtain
\begin{eqnarray*}
  && as_{\mathcal{A}_{L}}(x,y,z)-(-1)^{|x||y|}as_{\mathcal{A}_{L}}(y,x,z) \\
   &=&[R(x),[R(y),z]]+(-1)^{|x|(|y|+|z|)}[R(y),[z,R(x)]]-(R(x)\cdot R(y))\cdot z\\
   &+&(-1)^{|x||y|}z\cdot(R(x)\cdot R(y))-(-1)^{|z|(|x|+|y|)+|x||y|}z\cdot (R(y)\cdot R(x))\\
   &=&[R(x),[R(y),z]]+(-1)^{|x|(|y|+|z|)}[R(y),[z,R(x)]]+(-1)^{|z|(|x|+|y|)}[z,[R(x),R(y)]]\\
   &=&\circlearrowleft_{R(x),R(y),z} (-1)^{|x||z|}[R(x),[R(y),z]].
\end{eqnarray*}
Lie-admissibility now implies that
$$as_{\mathcal{A}_{L}}(x,y,z)-(-1)^{|x||y|}as_{\mathcal{A}_{L}}(y,x,z)=0,$$
showing that $\mathcal{A}_{L}$ is a pre-Lie superalgebra.
\end{proof}
\begin{thm}\label{RB+ass=preLie super} Let $(\mathcal{A},\mu,R)$ be a Rota-Baxter associative superalgebra of weight $\lambda=-1$. Define the even binary operation $"\circ"$ on any homogeneous element $x,y\in \mathcal{A}$ by
\begin{eqnarray}\label{ass+RB==pre-Lie}
x \circ y &=& \mu (R(x),y) -(-1)^{|x||y|} \mu(y,R(x))- \mu(x,y)\nonumber\\
          &=& R(x)y-(-1)^{|x||y|} y R(x)-x y.
\end{eqnarray}
Then $\mathcal{A}_{L}=(\mathcal{A},\circ)$ is a pre-Lie superalgebra.
\end{thm}
\begin{proof}For all homogeneous elements $x,y,z$ in $\mathcal{A}$, we have
\begin{eqnarray*}
  x\circ(y\circ z) &=& R(x)(R(y)z)-(-1)^{|y||z|}R(x)(z R(y))-R(x)(y z)\\
   &-& (-1)^{|x|(|y|+|z|)}(R(y)z)R(x)+(-1)^{|x|(|y|+|z|)+|y||z|}(z R(y))R(x)\\
   &+&(-1)^{|x|(|y|+|z|)}(y z)R(x)
   - x(R(y)z)+(-1)^{|y||z|}(z R(y))x+ x(y z),
  \end{eqnarray*}
and
\begin{eqnarray*}
  (x\circ y)\circ z &=& R(R(x)y)z-(-1)^{|x||y|}R(y R(x))z-R(x y)z- (-1)^{|z|(|x|+|y|)}(R(x)y)R(z) \\
   &+&(-1)^{|z|(|x|+|y|)+|x||y|}(y R(x))R(z)+(-1)^{|z|(|x|+|y|)}(x y)z\\
   &-& (R(x)y)z+(-1)^{|x||y|}(y R(x))z +(x y )z.
  \end{eqnarray*}
Then,  we obtain
\begin{eqnarray*}
  && as_{\mathcal{A}_{L}}(x,y,z)-(-1)^{|x||y|}as_{\mathcal{A}_{L}}(y,x,z) \\
   &=&x\circ(y\circ z)-(x\circ y)\circ z-(-1)^{|x||y|}y\circ (x \circ z)+(-1)^{|x||y|}(y \circ x)\circ z\\
   &=& R(x)(R(y)z)-(-1)^{|y||z|}R(x)(z R(y))-R(x)(y z)- (-1)^{|x|(|y|+|z|)}(R(y) z)R(x)\\
   &+&(-1)^{|x|(|y|+|z|)+|y||z|}(z R(y))R(x)+(-1)^{|x|(|y|+|z|)}(y z)R(x)-x(R(y)z)\\
   &+&(-1)^{|x|(|y|+|z|)+|y||z|}(z R(y))x+ x(y z)
   - R(R(x)y)z+(-1)^{|x||y|}R(y R(x))z+ R(x y)z \\
   &+& (-1)^{|z|(|x|+|y|)}(R(x)y) R(z)
   -(-1)^{|z|(|x|+|y|)+|x||y|}(y R(x))R(z)
   -(-1)^{|z|(|x|+|y|)}(x y) z\\
   &+& (R(x)y) z+(-1)^{|x||y|}(y R(x))z+(x y)z-(-1)^{|x||y|}R(y)(R(x)z)+(-1)^{|x|(|y|+|z|)}R(y)(z R(x))\\
   &+&(-1)^{|x||y|}R(y)(x z)+(-1)^{|y||z|}(R(x)z)R(y)-(-1)^{|z|(|x|+|y|)}(z R(x))R(y)-(-1)^{|y||z|}(x z)R(y)\\
   &+& (-1)^{|x||y|}y(R(x)z)-(-1)^{|z|(|x|+|y|)}(z R(x))y+(-1)^{|x||y|}y(x z)+(-1)^{|x||y|}R(R(y)x)z\\
   &-& R(x R(y))z-(-1)^{|x||y|}R(y x)z-(-1)^{|z|(|x|+|y|)+|x||y|}(R(y)x)R(z)
   +(-1)^{|z|(|x|+|y|)}(x R(y))R(z)\\
   &+&(-1)^{|z|(|x|+|y|)+|x||y|}(y x)z-(-1)^{|x||y|}(R(y)x)z+(x R(y))z+(-1)^{|x||y|}(y x)z.
   \end{eqnarray*}
The above sum vanishes by associativity and the Rota-Baxter identity (\ref{iden-Rota}) with $ \lambda=-1$.
\end{proof}
\begin{cor}Let $(\mathcal{A},\mu,R)$ be a Rota-Baxter associative superalgebra. Then $R$ is still a Rota-Baxter operator of weight $\lambda=1$ on the pre-Lie superalgebra $(\mathcal{A},\circ)$ defined in  (\ref{ass+RB==pre-Lie}).
\end{cor}

\section{$L$-dendriform superalgebras}
The notion of $L$-dendriform algebra was introduced by C. Bai, L. Liu and X. Ni in $2010$  $($see \cite{Bai-Liu L-dendrifom}$)$. In this section, we extend this notion to the graded case of $L$-dendriform superalgebra. Then we study  relationships between  associative superalgebras, $L$-dendriform superalgebras and pre-Lie superalgebras.  Moreover, we introduce the notion of Rota-Baxter operator $($of weight zero$)$ on the $\mathcal{A}$-bimodule and we provide construction of associative bimodules from bimodules over $L$-dendriform superalgebras and  a construction of $L$-dendriform bimodules from bimodules over pre-Lie superalgebras.
\subsection{L-dendriform superalgebras and associative superalgebras}
\subsubsection{Definition and some basic properties}

\begin{df}An $L$-dendriform superalgebra is a triple $(\mathcal{A},\rhd,\lhd)$ consisting of a $\mathbb{Z}_2$-graded vector space $\mathcal{A}$ and  two even bilinear maps
$\rhd,\lhd:\mathcal{A}\otimes \mathcal{A}\longrightarrow\mathcal{A}$
satisfying, for all homogeneous elements $x,y,z$ in $\mathcal{A}$, 
\begin{eqnarray}\label{1-Ldend}
x\rhd(y\rhd z)&=&(x\rhd y)\rhd z+(x\lhd y)\rhd z+(-1)^{|x||y|}y\rhd (x\rhd z)\\
&&-(-1)^{|x||y|}(y\lhd x)\rhd z-(-1)^{|x||y|}(y\rhd x)\rhd z,\nonumber
 \end{eqnarray}
 \begin{eqnarray}\label{2-Ldend}
x\rhd(y\lhd z)&=&(x\rhd y)\lhd z+(-1)^{|x||y|}y\lhd (x\rhd z)\\
&&+(-1)^{|x||y|}y\lhd (x \lhd z)-(-1)^{|x||y|}(y\lhd x)\lhd z.\nonumber
 \end{eqnarray}
The associated bracket to an $L$-dendriform superalgebra is defined as 
$[x,y]=x \rhd y -(-1)^{|x||y|}y \lhd x.$
\end{df}
\begin{df}\
\begin{enumerate}
\item Let $(\mathcal{A},\rhd,\lhd)$ be an $L$-dendriform superalgebra, $V$ be a $\mathbb{Z}_2$-graded vector space and $l_\rhd,r_\rhd,l_\lhd,r_\lhd:\mathcal{A}\longrightarrow End(V)$ be four even linear maps. The tuple $(V,l_\rhd,r_\rhd,l_\lhd,r_\lhd)$ is an $\mathcal{A}$-bimodule if for any homogeneous elements $x,y\in \mathcal{A}$ and $u,v\in V$,  the following identities are satisfied
\begin{enumerate}
\item  $ [l_\rhd(x),l_\rhd(y)]= l_\rhd([x,y]),$
\item    $ [l_\rhd(x),l_\lhd(y)]= l_\lhd(x \circ y)+(-1)^{|x||y|}l_\lhd(y)l_\lhd(x),$
\item    $ r_\rhd(x\rhd y)(v)=r_\rhd(y)r_\rhd(x)(v)+r_\rhd(y)r_\lhd(x)(v)+(-1)^{|x||v|}l_\rhd(x)r_\rhd(y)(v) -(-1)^{|x||v|}r_\rhd(y)l_\rhd(x)(v)-(-1)^{|x||v|}r_\rhd(y)l_\lhd(x)(v),$
\item   $  r_\rhd(x\lhd y)(v)= r_\lhd(y)r_\rhd(x)(v)-(-1)^{|x||v|}\Big(l_\lhd(x)r_\rhd(y)(v)-l_\lhd(x)r_\lhd(y)(v)+r_\lhd(y)l_\lhd(x)(v)\Big),$
\item    $ l_\rhd(x)r_\lhd(y)(v)-r_\lhd(y)l_\rhd(x)(v)=(-1)^{|x||v|} r_\lhd(x \bullet y)(v)-(-1)^{|x||v|}r_\lhd(y)r_\lhd(x)(v).$
\end{enumerate}
where $x \circ y=x\rhd y-(-1)^{|x||y|}y\lhd x$, and $x \bullet y=x\rhd y+x\lhd y$.\\
Moreover, The tuple $(V,\rhd_V,\lhd_V,l_\rhd,r_\rhd,l_\lhd,r_\lhd)$ is an $\mathcal{A}$-bimodule $\K$-superalgebra if  the following identities are satisfied
\begin{enumerate}
\item    $ l_\rhd(x)(u\rhd_V v)-(-1)^{|x||u|}u\rhd_V  l_\rhd(x)(v)=l_\rhd(x)(u)\rhd_V v+l_\lhd(x)(u)\rhd_V v$ 

 $  \hspace{3cm} -(-1)^{|x||u|}r_\rhd(x)(u)\rhd_V v-(-1)^{|x||u|}r_\lhd(x)(u)\rhd_V v,$
\item    $ l_\rhd(x)(u\lhd_V v)-(-1)^{|x||u|}u\lhd_V  l_\rhd(x)(v)=l_\rhd(x)(u)\lhd_V v-(-1)^{|x||u|}r_\lhd(x)(u)\lhd_V v$ 

 $  \hspace{3cm} +(-1)^{|x||u|}u\lhd_V l_\lhd(x)v,$
\item    $ u\rhd_V l_\lhd(x)(v)=r_\rhd(x)(u)\lhd_V v-(-1)^{|x||u|}l_\lhd(x)(u\rhd_V v)+(-1)^{|x||u|} l_\lhd(x)(u\lhd_V v)$ 

 $  \hspace{3cm}   -(-1)^{|x||u|}l_\lhd(x)(u)\lhd_V v.$
\end{enumerate}

\item Let $(\mathcal{A},\rhd,\lhd)$ be an $L$-dendriform superalgebra and $(V,\rhd_V,\lhd_V,l_\rhd,r_\rhd,l_\lhd,r_\lhd)$ be an $\mathcal{A}$-bimodule $\K$-superalgebra. An even linear map
$T:V\longrightarrow \mathcal{A}$ is called a super $\mathcal{O}$-operator of weight $\lambda \in \mathbb{K}$ associated to $(V,\rhd_V,\lhd_V,l_\rhd,r_\rhd,l_\lhd,r_\lhd)$ if $T$ satisfies for any homogeneous elements $u,v$ in $V$
\begin{eqnarray*}
   T(u)\rhd T(v)&=& T\Big(l_\rhd(T(u))v+(-1)^{|u||v|}r_\rhd(T(v))u+\lambda u\rhd_V v \Big),\\
    T(u)\lhd T(v)&=& T\Big(l_\lhd(T(u))v+(-1)^{|u||v|}r_\lhd(T(v))u +\lambda u\lhd_V v\Big).
\end{eqnarray*}
\end{enumerate}
In particular, a super $\mathcal{O}$-operator of weight $\lambda \in \mathbb{K}$ of the $L$-dendriform superalgebra $(\mathcal{A},\rhd,\lhd)$
associated to the bimodule $(\mathcal{A},L_\rhd,R_\rhd,L_\lhd,R_\lhd)$ is called a Rota-Baxter operator $($of weight $\lambda)$ on $(\mathcal{A},\rhd,\lhd)$, that is, $R$ satisfies for any homogeneous elements $x,y$ in $\mathcal{A}$
\begin{eqnarray*}
   R(x)\rhd R(y)&=& R\Big(R(x)\rhd y+R(x)\rhd u+\lambda x\rhd y\Big),\\
    R(x)\lhd R(y)&=& R\Big(R(x)\lhd y+R(x)\lhd y+ \lambda x\lhd y\Big).
\end{eqnarray*}
\end{df}
The following theorem provides a construction of $L$-dendriform superalgebras using super $\mathcal{O}$-operators of associative superalgebras.
\begin{thm}\label{ass=L-dendriform} Let $(\mathcal{A},\mu)$ be an associative superalgebra and $(V,l,r)$ be a $\mathcal{A}$-bimodule. If $T$ is a super $\mathcal{O}$-operator of weight zero associated with $(V,l,r)$, then there exists an $L$-dendriform superalgebra structure on $V$ defined by
\begin{equation}\label{th-operator1-ass}
    u \rhd v=l(T(u))v,~~u \lhd v= (-1)^{|u||v|}r(T(v))u,~~\forall~~u,v \in \mathcal{H}(V).
\end{equation}
\end{thm}
\begin{proof}For any homogeneous elements $u,v$ and $w$ in $V$, we have\\
$~~u\rhd (v\rhd w)= l(T(u))l(T(v))w,\hskip2.5cm (u\rhd v)\rhd w= l\Big(T(l(T(u))v)\Big)w$,\\
$~~(u\lhd v)\rhd w= (-1)^{|u||v|}l\Big(T(r(T(v))u)\Big)w,\hskip1cm (-1)^{|u||v|}v\rhd (u\rhd w)= (-1)^{|u||v|}l(T(v))l(T(u))w$,\\
$~~(-1)^{|u||v|}(v\lhd u)\rhd w= l\Big(T(r(T(u))v)\Big)w,\hskip1cm (-1)^{|u||v|}(v\rhd u)\rhd w= (-1)^{|u||v|}l\Big(T(l(T(v))u)\Big)w$.\\
Hence
 \begin{align*}
&   u\rhd (v\rhd w)-(u\rhd v)\rhd w-(u\lhd v)\rhd w-(-1)^{|u||v|}\Big(v\rhd (u\rhd w)-(v\lhd u)\rhd w-(v\rhd u)\rhd w\Big)\\
&  = l(T(u))l(T(v))w-l\Big(T(l(T(u))v)\Big)w-(-1)^{|u||v|}l\Big(T(r(T(v))u)\Big)w\\
&  -(-1)^{|u||v|}l(T(v))l(T(u))w+l\Big(T(r(T(u))v)\Big)w+(-1)^{|u||v|}l\Big(T(l(T(v))u)\Big)w\\
&   = (-1)^{|u||v|}l(\mu(T(v),T(u)))-(-1)^{|u||v|}l(T(v))l(T(u))w-l(\mu(T(u),T(v)))+l(T(u))l(T(v))w\\
&  = 0,
\end{align*}
and similarly, we have
\begin{eqnarray*}
 u\rhd(v\lhd w)-(u\rhd v)\lhd w-(-1)^{|u||v|}v\lhd(u\rhd w)-(-1)^{|u||v|}v\lhd(u\lhd w)+(-1)^{|u||v|}(v\lhd u)\lhd w= 0.
\end{eqnarray*}
Therefore $(V,\rhd,\lhd)$ is an $L$-dendriform superalgebra.
\end{proof}
A direct consequence of Theorem \ref{ass=L-dendriform} is the following construction of an $L$-dendriform superalgebra from a Rota-Baxter operator $($of weight zero$)$ of an associative superalgebra.
\begin{cor}Let $(\mathcal{A},\mu,R)$ be a Rota-Baxter associative superalgebra of weight zero. Then, the even binary operations given by
$$ x\rhd y=\mu(R(x),y),~~~~x\lhd y=\mu(x,R(y)),~~~~\forall~~x,y\in \mathcal{H}(\mathcal{A})$$
defines an $L$-dendriform superalgebra structure on $\mathcal{A}$.
\end{cor}
\begin{df}Let $(\mathcal{A},\rhd,\lhd)$ be an $L$-dendriform superalgebra and $R:\mathcal{A}\longrightarrow \mathcal{A}$ be a Rota-Baxter operator of weight zero. A Rota-Baxter operator on $\mathcal{A}$-bimodule $V$ $($relative to $R)$ is a map $R_V:V\longrightarrow V$ such that for all
homogeneous elements $x$ in $\mathcal{A}$ and $v$ in $V$
\begin{eqnarray*}
    & & \ \ \ R(x)\rhd R_V(v) = R_V \Big(R(x)\rhd v+x \rhd R_V(v)\Big), \ \
     R_V(v)\rhd R(x) = R_V \Big(R_V(v)\rhd x+v \rhd R(x)\Big),\\
    & & \ \ \ R(x)\lhd R_V(v) = R_V \Big(R(x)\lhd v+x \lhd R_V(v)\Big), \ \
    R_V(v)\lhd R(x) = R_V \Big(R_V(v)\lhd x+v \lhd R(x)\Big).
    \end{eqnarray*}
\end{df}
\begin{prop}Let $(\mathcal{A},\mu)$ be an associative superalgebra, $R:\mathcal{A}\longrightarrow \mathcal{A}$ a Rota-Baxter operator on $\mathcal{A}$, $V$  an $\mathcal{A}$-bimodule and $R_V$ a Rota-Baxter operator on $V$. Define a new actions of $\mathcal{A}$ on $V$ by
$$ x\rhd v=\mu(R(x),v),~~v\rhd x=\mu(R_V(v),x),~~x\lhd v=\mu(x,R_V(v)),~~v\lhd x=\mu(v,R(x)).$$
Equipped with these actions, $V$ becomes  an $\mathcal{A}$-bimodule over the associated  $L$-dendriform superalgebra.
\end{prop}
\begin{cor}Let $(V,l_\rhd,r_\rhd,l_\lhd,r_\lhd)$ be an $ \mathcal{A}$-bimodule of a dendriform superalgebra $(\mathcal{A},\rhd,\lhd)$. Let $(\mathcal{A},\mu)$ be the
associated associative superalgebra. If $T$ is a super $\mathcal{O}$-operator associated to $(V,l_\rhd,r_\rhd,l_\lhd,r_\lhd)$, then $T$ is a super $\mathcal{O}$-operator of $(\mathcal{A},\mu)$ associated to $(V,l_\rhd+l_\lhd,r_\rhd+r_\lhd)$.
\end{cor}
\subsection{$L$-dendriform superalgebras and pre-Lie superalgebras}
We have the following observation.
\begin{prop}Let $(\mathcal{A},\rhd,\lhd)$ be an $L$-dendriform superalgebra
\begin{enumerate}
\item The even binary operation $\circ:\mathcal{A}\otimes \mathcal{A}\longrightarrow \mathcal{A}$ given by
$$ x\circ y=x\rhd y-(-1)^{|x||y|}y \lhd x,~~\forall~~x,y \in \mathcal{H}(\mathcal{A})$$
defines a pre-Lie superalgebra $(\mathcal{A},\circ)$ which is called the associated vertical pre-Lie superalgebra of $(\mathcal{A},\rhd,\lhd)$ and $(\mathcal{A},\rhd,\lhd)$ is called a compatible L-dendriform superalgebra structure on the pre-Lie superalgebra $(\mathcal{A},\circ)$.
\item The even binary operation $\bullet:\mathcal{A}\otimes \mathcal{A}\longrightarrow \mathcal{A}$ given by
$$ x\bullet y=x\rhd y+x \lhd y,~~\forall~~x,y \in \mathcal{H}(\mathcal{A})$$
defines a pre-Lie superalgebra $(\mathcal{A},\bullet)$ which is called the associated horizontal pre-Lie superalgebra of $(\mathcal{A},\rhd,\lhd)$ and $(\mathcal{A},\rhd,\lhd)$ is called a compatible L-dendriform superalgebra structure on the pre-Lie superalgebra $(\mathcal{A},\bullet)$.
\item Both $(\mathcal{A},\circ)$ and $(\mathcal{A},\bullet)$ have the same sub-adjacent Lie superalgebra $g(\mathcal{A})$ defined by
$$[x,y]=x\rhd y+x\lhd y-(-1)^{|x||y|}y \rhd x-(-1)^{|x||y|}y \lhd x,~~~~\forall~~x,y\in \mathcal{H}(\mathcal{A}).$$
\end{enumerate}
\end{prop}
\begin{proof}Straightforward.
\end{proof}
\begin{cor}Let $(V,l_\rhd,r_\rhd,l_\lhd,r_\lhd)$ be a bimodule of an $L$-dendriform superalgebra $(\mathcal{A},\rhd,\lhd)$. Let $(\mathcal{A},\circ)$ be the
associated pre-Lie superalgebra. If $T$ is a super $\mathcal{O}$-operator associated to $(V,l_\rhd,r_\rhd,l_\lhd,r_\lhd)$, then $T$ is a super $\mathcal{O}$-operator of $(\mathcal{A},\circ)$ associated to $(V,l',r')$, where $l'=l_\rhd+(-1)^{|u||v|}r_\lhd$ and $r'=l_\lhd+r_\rhd$.
\end{cor}
Conversely, we can construct $L$-dendriform superalgebras from $\mathcal{O}$-operators of pre-Lie superalgebras.
\begin{thm}\label{preLie=L-dendriform}Let $(\mathcal{A},\circ)$ be a pre-Lie superalgebra and $(V,l,r)$ be an $\mathcal{A}$-bimodule. If $T$ is a super $\mathcal{O}$-operator of weight zero associated to $(V,l,r)$, then there exists an $L$-dendriform superalgebra structure on $V$ defined by
\begin{equation}\label{th-operator1}
    u \rhd v=l(T(u))v,~~u \lhd v= -r(T(u))v,~~\forall~~u,v \in \mathcal{H}(V).
\end{equation}
Therefore, there is a pre-Lie superalgebra structure on $V$ defined by
\begin{equation}\label{th-operator2}
    u \circ v= u \rhd v-(-1)^{|u||v|}v \lhd u,~~\forall~~u,v \in \mathcal{H}(V)
\end{equation}
as the associated vertical pre-Lie superalgebra of $(V,\rhd,\lhd)$ and $T$ is a homomorphism of pre-Lie superalgebra.\\
Furthermore, $T(V)=\{T(v)~~/~~ v\in V\}\subset \mathcal{A}$ is a pre-Lie subsuperalgebra of $(\mathcal{A},\circ)$ and there is an $L$-dendriform superalgebra structure on $T(V)$ given by
\begin{equation}\label{th-operator3}
    T(u)\rhd T(v)= T(u \rhd v),~~ T(u)\lhd T(v)= T(u \lhd v),~~\forall~~u,v \in \mathcal{H}(V).
\end{equation}
Moreover, the corresponding associated vertical pre-Lie superalgebra structure on $T(V)$ is a pre-Lie subsuperalgebra of $(\mathcal{A},\circ)$ and $T$ is an homomorphism of $L$-dendriform superalgebra.
\end{thm}
\begin{proof}For any homogeneous elements $u,v$ and $w$ in $V$, we have\\
$~~u\rhd (v\rhd w)= l(T(u))l(T(v))w,\hskip4cm (u\rhd v)\rhd w= l\Big(T(l(T(u))v)\Big)w$,\\
$~~(u\lhd v)\rhd w= -l\Big(T(r(T(u))v)\Big)w,\hskip3.7cm (-1)^{|u||v|}v\rhd (u\rhd w)= (-1)^{|u||v|}l(T(v))l(T(u))w$,\\
$~~(-1)^{|u||v|}(v\lhd u)\rhd w= -(-1)^{|u||v|}l\Big(T(r(T(v))u)\Big)w,\hskip1cm (-1)^{|u||v|}(v\rhd u)\rhd w= (-1)^{|u||v|}l\Big(T(l(T(v))u)\Big)w$,\\
$~~u\rhd(v\lhd w)=-l(T(u))r(T(v))w, \hskip3.7cm (u\rhd v)\lhd w= r\Big(T(l(T(u))v)\Big)w,$\\
$~~(-1)^{|u||v|}v\lhd(u\rhd w)=-(-1)^{|u||v|}r(T(v))l(T(u))w, \hskip1cm (-1)^{|u||v|}v\lhd(u\lhd w)=(-1)^{|u||v|}r(T(v))r(T(u)),$\\
$~~(-1)^{|u||v|}(v\lhd u)\lhd w=(-1)^{|u||v|}r\Big(T(r(T(v))u)\Big)w.$\\
Hence
\begin{align*}
&      u\rhd (v\rhd w)-(u\rhd v)\rhd w-(u\lhd v)\rhd w-(-1)^{|u||v|}v\rhd (u\rhd w)+(-1)^{|u||v|}(v\lhd u)\rhd w+(-1)^{|u||v|}(v\rhd u)\rhd w\\
&    = l(T(u))l(T(v))w -(-1)^{|u||v|}l(T(v))l(T(u))w-l(T(l(T(u)))v)w+l(T(r(T(u))v))w\\
&    - (-1)^{|u||v|}l(T(r(T(v))u))w+(-1)^{|u||v|}l(T(l(T(v))u))w\\
&  = l(T(u))l(T(v))w -(-1)^{|u||v|}l(T(v))l(T(u))w-l\Big(T(u)\circ T(v)\Big)w+(-1)^{|u||v|}l\Big(T(v)\circ T(u)\Big)w\\
&  \  = 0,
\end{align*}
and
\begin{align*}
&    u\rhd(v\lhd w)-(u\rhd v)\lhd w-(-1)^{|u||v|}v\lhd(u\rhd w)-(-1)^{|u||v|}v\lhd(u\lhd w)+(-1)^{|u||v|}(v\lhd u)\lhd w\\
&   = -l(T(u))r(T(v))w+(-1)^{|u||v|}r(T(v))l(T(u))w+ r(T(u)\circ T(v))w-(-1)^{|u||v|}r(T(v))r(T(u))\\
&   = 0.
\end{align*}
Therefore, $(V,\rhd,\lhd)$ is an $L$-dendriform superalgebra. The other conditions follow easily.
\end{proof}
A direct consequence of Theorem \ref{preLie=L-dendriform} is the following construction of an $L$-dendriform superalgebra from a Rota-Baxter operator $($of weight zero$)$ of a pre-Lie superalgebra.
\begin{cor}\label{RB-pre==L-den}Let $(\mathcal{A},\circ)$ be a pre-Lie superalgebra and $R$ be a Rota-Baxter operator on $\mathcal{A}$ $($of weight zero$)$. Then the even binary operations given by
\begin{equation}\label{RB-pre=Ldend}
 x\rhd y=R(x)\circ y,~~~~x\lhd y=-(-1)^{|x||y|}y\circ R(x)
 \end{equation}
defines an $L$-dendriform superalgebra structure on $\mathcal{A}$.
\end{cor}
\begin{lem}Let $\{R_1,R_2\}$ be a pair of commuting Rota-Baxter operators $($of weight zero$)$ on a pre-Lie superalgebra $(\mathcal{A},\circ)$. Then $R_2$ is a
Rota-Baxter operator $($of weight zero$)$ on the $L$-dendriform superalgebra $(\mathcal{A},\rhd,\lhd)$ defined in \eqref{RB-pre=Ldend} with $R=R_1$.
\end{lem}
\begin{thm}Let $(\mathcal{A},\circ)$ be a pre-Lie superalgebra. Then there exists a compatible $L$-dendriform superalgebra structure on $(\mathcal{A},\circ)$ such that $(\mathcal{A},\circ)$ is the associated vertical pre-Lie superalgebra if and only if there exists an invertible super $\mathcal{O}$-operator $($of weight zero$)$ of $(\mathcal{A},\circ)$.
\end{thm}
\begin{proof}Straightforward.
\end{proof}
Next, we provide   a construction of an $L$-dendriform bimodule from a bimodule over a pre-Lie superalgebra.
\begin{prop}Let $(\mathcal{A},\circ,R)$ be a Rota-Baxter pre-Lie superalgebra of weight zero, $V$ an $\mathcal{A}$-bimodule and $R_V$ a Rota-Baxter operator on $V$. Define new actions of $\mathcal{A}$ on $V$ by
$$ x\rhd v=R(x)\circ v,~~~~v\rhd x=R_V(v)\circ x,~~~~x\lhd v=-(-1)^{|x||v|}v \circ R(x),~~v\lhd x=-(-1)^{|x||v|}x\circ R_V(v).$$
Equipped with actions, $V$ is a bimodule over the $L$-dendriform superalgebra of Corollary \ref{RB-pre==L-den}.
\end{prop}
\begin{proof}Let $x,y$ be homogeneous elements in $\mathcal{A}$ and $v$ in $V$. Then, we have
\begin{eqnarray*}
& & \ \ \   r_\lhd(x\bullet y)(v)-(-1)^{|x||y|}r_\lhd(y)r_\lhd(x)(v)-[l_\rhd(x),r_\lhd(y)](v)\\
& & \ \ \ = (-1)^{|v|(|x|+|y|)}v\lhd (x\rhd y+x\lhd y)-(-1)^{|v|(|x|+|y|)}(v\lhd x)\lhd y\\
& & \ \ \ \hskip0.6cm -(-1)^{|y||v|}x\rhd(v\lhd y)+(-1)^{|v|(|x|+|y|)}(x\rhd v)\lhd y\\
& & \ \ \ =-(R(x)\circ y)\circ R_V(v)+(-1)^{|x||y|}(y\circ R(x))\circ R_V(v)-(-1)^{|x||y|}y\circ R_V(x\circ R_V(v))\\
& & \ \ \ \hskip0.6cm -(-1)^{|x||y|}y\circ R_V(R(x)\circ v)+R(x)\circ (y\circ R_V(v))\\
& & \ \ \ = -(R(x)\circ y)\circ R_V(v)+R(x)\circ (y\circ R_V(v))-(-1)^{|x||y|}y\circ (R(x)\circ R_V(v))\\
& & \ \ \ \hskip0.6cm +(-1)^{|x||y|}(y\circ R(x))\circ R_V(v)\\
& & \ \ \ = 0.
\end{eqnarray*}
Therefore  \begin{eqnarray*}[l_\rhd(x),r_\lhd(y)](v)&=& r_\lhd(x\bullet y)(v)-(-1)^{|x||y|}r_\lhd(y)r_\lhd(x)(v).\end{eqnarray*}
Similarly, we have
\begin{eqnarray*}
& & \ \ \  r_\rhd(x\lhd y)(v)-(-1)^{|x||y|}r_\lhd(y)r_\rhd(x)(v)-l_\lhd(x)r_\rhd(y)(v)-[l_\lhd(x),r_\lhd(y)](v)\\
& & \ \ \ = (-1)^{|v|(|x|+|y|)}v\rhd(x\lhd y)-(-1)^{|v|(|x|+|y|)+|x||y|}(v\rhd x)\lhd y-(-1)^{|y||v|}x\lhd (v\rhd y)\\
& & \ \ \ \hskip0.6cm -(-1)^{|y||v|}x\lhd(v\lhd y)+(-1)^{|y||v|}(x\lhd v)\lhd y\\
& & \ \ \ = -(-1)^{|v|(|x|+|y|)+|x||y|}R_V(v)\circ(y\circ R(x))+ (-1)^{|x|(|y|+|v|)}y\circ R_V(R_V(v)\circ x)\\
& & \ \ \ \hskip0.6cm +(-1)^{|v|(|x|+|y|)+|x||y|}(R_V(v)\circ y)\circ R(x)-(-1)^{|x|(|y|+|v|)}(y\circ R_V(v))\circ R(x)\\
& & \ \ \ \hskip0.6cm +(-1)^{|x|(|y|+|v|)} y\circ R_V(v\circ R(x)) \\
& & \ \ \ = (-1)^{|x|(|y|+|v|)} \Big(y\circ (R_V(v)\circ R(x))-(y\circ R_V(v))\circ R(x)\Big)\\
& & \ \ \ \hskip0.6cm +(-1)^{|v|(|x|+|y|)+|x||y|}\Big(R_V(v)\circ y)\circ R(x)-R_V(v)\circ(y\circ R(x))\Big)\\
& & \ \ \ = 0,
\end{eqnarray*}
then  \begin{eqnarray*}r_\rhd(x\lhd y)(v)&=&(-1)^{|x||y|}r_\lhd(y)r_\rhd(x)(v)+l_\lhd(x)r_\rhd(y)(v)+[l_\lhd(x),r_\lhd(y)](v).\end{eqnarray*}
The others axioms are similar. Therefore $(V,l_\rhd,r_\rhd,l_\lhd,r_\lhd)$ is a bimodule over the $L$-dendriform superalgebra $(\mathcal{A},\rhd,\lhd)$.
\end{proof}
\section{Rota-Baxter operators on $2$-dimensional pre-Lie superalgebras}
The purpose of this section is to compute all Rota-Baxter operators $($of weight zero$)$ on the two-dimensional complex pre-Lie superalgebras given by R. Zhang and C. Bai (see \cite{Bai and Zhang classif}).
In the following, let $\mathbb{C}$ be the ground field of complex numbers and $\{e_1,e_2\}$ be a homogeneous basis of a pre-Lie superalgebra $(\mathcal{A},\circ)$, where  $\{e_1\}$ is a basis of $\mathcal{A}_0$ and $\{e_2\}$ is a basis of $\mathcal{A}_1$.\\
By direct computation and by help of  a computer algebra system, we obtain the following results.
\begin{prop}The Rota-Baxter operators $($of weight zero$)$ on $2$-dimensional pre-Lie superalgebras $($associative or non-associative$)$ of type $B_{1},~~B_{2}$ and $B_{3}$ are given as follows:
\end{prop}
\subsection{\textbf{Rota-Baxter operators on pre-Lie superalgebras of type }$B_{1}$} $ $

$\blacksquare$~~\textsf{The pre-Lie superalgebra $(B_{1,1},\circ):\ \ e_2\circ e_1=e_2$} has the Rota-Baxter operator defined as
$$R_1(e_1)=a_1e_1,\ R_1(e_2)= 0.$$
$\blacksquare$~~\textsf{The pre-Lie superalgebras of type}
$\left\{ \begin{array}{ll}
\ \ (B_{1,2},\circ)\ \ \ : \  \ \ \ e_1\circ e_1=e_1,~~~~e_2\circ e_1=e_2~~\text{(associative)}. & \hbox{} \\
\Big((B_{1,3})_k,\circ \Big): \ \ \ \ e_1\circ e_1=ke_1,~~~~e_2\circ e_1=e_2,~~k\neq 0,1. & \hbox{} \\
\ \ (B_{1,4},\circ)\ \ \ : \ \ \ \ e_1\circ e_1=e_1,~~~~e_1\circ e_2=e_2~~\text{(associative)}. & \hbox{} \\
\Big((B_{1,5})_k,\circ \Big):\ \ \ \ e_1\circ e_1=ke_1,~~~~e_1\circ e_2=~ke_2,~~~~ e_2\circ e_1=(k+1)e_2,~~k\neq 0,-1. & \hbox{}
                                                            \end{array}
                                                          \right.$
 have only the trivial Rota-Baxter operator, that is,
\begin{itemize}
   \item[]  $R_1(e_1)=0,~~ R_1(e_2)= 0.$
   \end{itemize}
\subsection{\textbf{Rota-Baxter operators on pre-Lie superalgebras of type $B_{2}$ }}$ $

$\blacksquare$~~\textsf{The pre-Lie superalgebra} $(B_{2,1},\circ)$~~(associative).
$$e_2\circ e_2=\frac{1}{2}e_1.$$
Rota-Baxter operators $RB(B_{2,1})$ are:
\begin{itemize}
   \item[] $R_1(e_1)= a_{1}e_1,~~ R_1(e_2)= 0.$
   \item[]  $R_2(e_1)=a_{1}e_1 ,~~ R_2(e_2)= 2a_{1} e_2.$
   \end{itemize}
$\blacksquare$ \textsf{The pre-Lie superalgebra $(B_{2,2},\circ)$} (associative)
$$ e_1\circ e_1=e_1,~~e_1\circ e_2= e_2,~~e_2\circ e_1=e_2,~~e_2\circ e_2=\frac{1}{2}e_1. $$
Rota-Baxter operators $RB(B_{2,2})$ is:
\begin{eqnarray*}
    R_1(e_1)&=& 0,~~ R_1(e_2)=0.
   \end{eqnarray*}
\subsection{\textbf{Rota-Baxter operators on pre-Lie superalgebras of type} $B_{3}$} $ $

$\blacksquare$ \textsf{The pre-Lie superalgebra $(B_{3,1},\circ)$} (associative).
$$e_i \circ e_j=0,~~\forall~~i,j=1,2,3.$$
Rota-Baxter operators $RB(B_{3,1})$ are:
\begin{itemize}
  \item[]  $R_1(e_1)= a_1e_1,\ R_1(e_2)= a_2 e_2. $
   \end{itemize}
$\blacksquare$ \textsf{The pre-Lie superalgebra $(B_{3,2},\circ)$} (associative)
$$ e_1\circ e_1=e_1. $$
Rota-Baxter operators $RB(B_{3,2})$ are:
\begin{itemize}
   \item[] $R_1(e_1)= 0,\ R_1(e_2)=a_1 e_2.$
   \end{itemize}
$\blacksquare$ \textsf{The pre-Lie superalgebra $(B_{3,3},\circ)$} (associative)
$$ e_1\circ e_1=e_1,~~e_1\circ e_2=e_2,~~e_2\circ e_1=e_2. $$
Rota-Baxter operator $RB(B_{3,3})$ is:
\begin{itemize}
  \item[]  $R_1(e_1)=0,\ R_1(e_2)=0.$
   \end{itemize}
\section{ Rota-Baxter operators on $3$-dimensional pre-Lie superalgebras}
\subsection{ Rota-Baxter operators on $3$-dimensional pre-Lie superalgebras with $2$-dimensional odd part}$ $
We still work over the ground field $\mathbb{C}$ of complex numbers. Using the classification of the $3$-dimensional pre-Lie superalgebras with one dimensional even part were given by R. Zhang and C. Bai (see \cite{Bai and Zhang classif}). The purpose of this section is to provide, using Definition \ref{defi-oper-pre}, all Rota-Baxter operators $($of weight zero$)$ on these pre-Lie superalgebras by direct computation.
In the following, let $\{e_1,e_2,e_3\}$ be a homogeneous basis of a pre-Lie superalgebra $(\mathcal{A},\circ)$, where  $\{e_1\}$ is a basis of $\mathcal{A}_0$ and $\{e_2,e_3\}$ is a basis of $\mathcal{A}_1$.\\
%We have now the classification of all Rota-Baxter operators $($of weight zero$)$ on the $3$-dimensional complex pre-Lie superalgebras. The following results can be obtained by direct computation.
\begin{prop}The Rota-Baxter operators $($of weight zero$)$ on $3$-dimensional pre-Lie superalgebras $($associative or non-associative$)$ with $2$-dimensional odd part of type $C_1,~~C_{2_h},~~C_{3},~~C_{4},~~C_{5}$ and $C_6$ are given as follows:
\end{prop}
\subsubsection{\textbf{Rota-Baxter operators on pre-Lie superalgebras of type} $C_{1}$}$ $

$\blacksquare$ \textsf{The pre-Lie superalgebra $(C_{1,1},\circ)$}
$$ e_2\circ e_3=-e_1,~~e_3\circ e_1=e_2,~~e_3\circ e_2=e_1. $$
Rota-Baxter operators $RB(C_{1,1})$ are:
\begin{itemize}
  \item[]  $R_1(e_1)=0,\ R_1(e_2)=0,\ R_1(e_3)=a_1e_2+e_3.$
  \item[]  $R_2(e_1)=0,\ R_2(e_2)=a_2e_2,\ R_2(e_3)=a_1 e_2,~~a_2\neq 0.$
  \item[]  $R_3(e_1)=a_3e_1,\ R_3(e_2)=0,\ R_3(e_3)=a_1 e_2.$
   \end{itemize}
$\blacksquare$ \textsf{The pre-Lie superalgebra $\Big((C_{1,2})_k,\circ\Big)$} (associative)
$$ e_1\circ e_3=ke_2,~~e_3\circ e_1=(k+1)e_2.$$
Rota-Baxter operators $RB((C_{1,2})_k)$ are:
\begin{itemize}
  \item[]  $R_1(e_1)=0,\ R_1(e_2)=a_1 e_2,\ R_1(e_3)=a_2 e_2,~~k\neq -1.$
  \item[]  $R_2(e_1)=0,\ R_2(e_2)= 0,\ R_2(e_3)=a_2 e_2+a_3 e_3,~~k= -1.$
  \item[]  $R_3(e_1)=a_4e_1,\ R_3(e_2)= a_1 e_2,\ R_3(e_3)=a_2 e_2+\frac{a_1a_4}{a_4-a_1}e_3,~~a_1\neq a_4,~~k=-1.$
   \end{itemize}
$\blacksquare$ \textsf{The pre-Lie superalgebra $(C_{1,3},\circ)$:}
$$ e_1\circ e_1=e_1,~~e_3\circ e_1=e_2. $$
Rota-Baxter operators $RB(C_{1,3})$ are:
\begin{itemize}
 \item[] $R_1(e_1)=0,\ R_1(e_2)=0,\ R_1(e_3)=a_1 e_2+a_2e_3.$
 \item[] $R_2(e_1)=0,\ R_2(e_2)=a_3e_2,\ R_2(e_3)=a_1 e_2.$
   \end{itemize}
$\blacksquare $ \textsf{The pre-Lie superalgebra $(C_{1,4},\circ)$:}
$$ e_1\circ e_1=e_1,~~e_1\circ e_2=e_2,~~e_1\circ e_3=e_3,~~e_2\circ e_1=e_2,~~e_3\circ e_1=e_2+ e_3. $$
Rota-Baxter operators $RB(C_{1,4})$ are:
\begin{itemize}
  \item[] $R_1(e_1)=0,\ R_1(e_3)=0,\ R_1(e_2)=a_1 e_2.$
  \item[] $R_2(e_1)=0,\ R_2(e_2)=0,\ R_2(e_3)=0.$
  \end{itemize}
\subsubsection{\textbf{ Rota-Baxter operators on pre-Lie superalgebras of type $C_{2_{h}}$}}$ $

$\blacksquare $ \textsf{The pre-Lie superalgebra $(C_{2_{h},1},\circ)$:}
$$ e_1\circ e_1=(h+1)e_1,~~e_2\circ e_1=e_2,~~e_2\circ e_3=-e_1,~~e_3\circ e_1=h e_3,~~~~ e_3\circ e_2= e_1,~~h \in \mathbb{C}.$$
Rota-Baxter operators $RB(C_{2_{h},1})$ are:\\
$\triangleright$ \textbf{~~Case 1}: If $h=0$, we have
\begin{itemize}
\item[] $R_1(e_1)=0,\ R_1(e_2)=0,\ R_1(e_3)=a_1 e_2+a_2e_3.$
 \item[] $R_2(e_1)=0,\ R_2(e_2)= a_3e_3,\ R_2(e_3)=a_2e_3.$
 \item[] $R_3(e_1)=0,\ R_3(e_2)= 0,\ R_3(e_3)=a_2e_3.$
 \item[] $R_4(e_1)= 0,\ R_4(e_2)=a_3e_3,\ R_4(e_3)=a_2e_3,~~~a_3\neq 0.$
   \end{itemize}
$\triangleright$ \textbf{~~Case 2}: If $h\in \mathbb{C}^\ast$, we have
\begin{itemize}
 \item[]  $R_5(e_1) =0,\ R_5(e_2)=0,\ R_5(e_3)=a_1e_2.$
 \item[] $R_6(e_1) =0,\ R_6(e_2)=a_3e_3,\ R_6(e_3)=0,~~a_{3}\neq 0.$
 \item[] $R_7(e_1) =0,\ R_7(e_2)=a_5e_2+a_3e_3,\ R_7(e_3)=-\frac{a_5^2}{h a_3}e_2-\frac{a_5}{h}e_3,~~a_3\neq 0,~~a_5\neq 0.$
 \item[]  $R_8(e_1) =0,\ R_8(e_2)=0,\ R_8(e_3)=0.$
\end{itemize}
$\triangleright$ \textbf{~~Case 3}: If $h=-1$, we have
\begin{itemize}
 \item[]$R_{9}(e_1) =0,\ R_{9}(e_2)=a_3e_3,\ R_{9}(e_3)=0.$
  \item[]$R_{10}(e_1) =0,\ R_{10}(e_2)=a_5e_2+a_3e_3,\ R_{10}(e_3)=\frac{a_5^2}{a_3}e_3+a_5e_3,~~a_3\neq 0.$
  \item[]$R_{11}(e_1) =a_4 e_1,\ R_{11}(e_2)=0,\ R_{11}(e_3)=a_1e_2,~~a_4\neq 0.$
 \item[] $R_{12}(e_1) = a_4e_1,\ R_{12}(e_2)=0,\ R_{12}(e_3)=0,~~a_4\neq 0.$
\end{itemize}
$\blacksquare $ \textsf{The pre-Lie superalgebra $(C_{2_{h},2},\circ)$:}
~$$ e_1\circ e_1=(1-h)e_1,~~e_1\circ e_3=e_2,~~e_2\circ e_1=e_2,~~e_3\circ e_1= e_2+h e_3,~~h \in \mathbb{C}. $$
Rota-Baxter operators $RB(C_{2_{h},2})$ are:\\
$\triangleright$~~ \textbf{~~Case 1}: If $h=0$, we have
\begin{itemize}
  \item[] $R_{1}(e_1) =0,\ R_{1}(e_2)=0,\ R_{1}(e_3)=a_{1}e_2+a_{2}e_3.$
  \end{itemize}
$\triangleright$~~ \textbf{~~Case 2}: If $h\in \mathbb{C}^\ast$, we have
  \begin{itemize}
  \item[] $R_{2}(e_1) =0,\ R_{2}(e_2)=0,\ R_{2}(e_3)=a_{1}e_2,~~a_{1}\neq 0.$
  \item[] $R_{3}(e_1) =0,\ R_{3}(e_2)=e_2,\ R_{3}(e_3)=0.$
 \end{itemize}
$\triangleright$~~ \textbf{~~Case 3}: If $h=1$, we have
  \begin{itemize}
  \item[] $R_{4}(e_1) =a_{3} e_1,\ R_{4}(e_2)=0,\ R_{4}(e_3)=a_{1}e_2,~~a_{3}\neq 0.$
  \item[] $R_{5}(e_1) =0,\ R_{5}(e_2)=0,\ R_{5}(e_3)=a_{1} e_2,~~a_{1}\neq 0.$
  \item[] $R_{6}(e_1) =0,\ R_{6}(e_2)=e_2,\ R_{6}(e_3)=0.$
 \end{itemize}
$\blacksquare $ \textsf{The pre-Lie superalgebra $(C_{2_{h},3},\circ)$:}
$$ e_1\circ e_1=(1-h)e_1,~~e_1\circ e_2= (1-h) e_2+e_3,~~e_1\circ e_3=(1-h) e_3,~~e_2\circ e_1=(2-h) e_2+e_3 ,~~e_3\circ e_1=e_3,~~h\neq  1. $$
Rota-Baxter operators $RB(C_{2_{h},3})$ are:\\
$\triangleright$~~ \textbf{~~Case 1}: If $h=(-1)^{\frac{1}{3}}$ or $h=-(-1)^{\frac{2}{3}}$, we have
\begin{itemize}
 \item[] $R_{1}(e_1) = 0,\ R_{1}(e_2)=0,\ R_{1}(e_3)=a_1e_3.$
 \item[] $R_{2}(e_1) = 0,\ R_{2}(e_2)=a_2 e_2+a_3e_3,\ R_{2}(e_3)=a_2(h-1)e_2+a_2(h-1)e_3.$
  \end{itemize}
 $\triangleright$~~ \textbf{~~Case 2}: If $h\neq 1$, we have
   \begin{itemize}
 \item[] $R_{3}(e_1) =0,\ R_{3}(e_2)=0,\ R_{3}(e_3)=0.$
 \item[] $R_{4}(e_1) =0,\ R_{4}(e_2)=0,\ R_{4}(e_3)=a_1e_3,~~a_1\neq 0.$
 \item[] $R_{5}(e_1) =0,\ R_{5}(e_2)=0,\ R_{5}(e_3)=a_1e_3,~~h^3-2h^2+2h-1\neq 0.$
 \end{itemize}
$\blacksquare $ \textsf{The pre-Lie superalgebra $(C_{2_{h},4},\circ)$:}
$$ e_1\circ e_1=(h-1)e_1,~~e_1\circ e_2=e_3,~~e_2\circ e_1=e_2+e_3,~~e_3\circ e_1=h e_3,~~h\neq \pm 1.$$
Rota-Baxter operators $RB(C_{2_{h},4})$ are:
\begin{itemize}
 \item[] $R_{1}(e_1) =0,\ R_{1}(e_2)=a_1e_3,\ R_{1}(e_3)=a_2e_3,~~h=0.$
 \item[] $R_{2}(e_1) =0,\ R_{2}(e_2)=a_1e_3,\ R_{2}(e_3)=0,~~a_1\neq 0,~~h\neq 0.$
 \item[] $R_{3}(e_1) =0,\ R_{3}(e_2)=0,\ R_{3}(e_3)=0,~~h\neq 0.$
 \end{itemize}
$\blacksquare $ \textsf{The pre-Lie superalgebra $(C_{2_{h},5},\circ)$:}
$$ e_1\circ e_1=(1-h)e_1,~~e_1\circ e_2=(1-h)e_2+e_3,~~e_1\circ e_3=(1-h)e_3,~~e_2\circ e_1=(2-h)e_2+e_3,~~e_3\circ e_1=e_3,~~h\neq \pm 1. $$
Rota-Baxter operators $RB(C_{2_{h},5})$ are:
\begin{itemize}
\item[]  $R_{1}(e_1) =0,\ R_{1}(e_2)=e_2,\ R_{1}(e_3)=a_{1}e_2-a_{1}e_3,~~h=0.$
\item[]  $R_{2}(e_1) = 0,\ R_{2}(e_2)=0,\ R_{2}(e_3)=0,~~h\neq 0.$
\item[]  $R_{3}(e_1) =0,\ R_{3}(e_2)=a_{2} e_2-\frac{a_{3}}{(h-1)^2}e_3,\ R_{3}(e_3)=\frac{a_{2}(h-1)^2}{h}e_2-\frac{a_{2}}{h}e_3,~~a_{2}\neq 0,~~h\neq 0.$
\item[]  $R_{4}(e_1) =0,\ R_{4}(e_2)=a_{2}e_2-a_{2}e_3,\ R_{4}(e_3)= \frac{a_{2}}{2}e_2-\frac{a_{2}}{2}e_3,~~a_{2}\neq 0,~~h=2.$
\item[]  $R_{5}(e_1) =0,\ R_{5}(e_2)=a_{3}e_3,\ R_{5}(e_3)=0,~~a_{3}\neq 0.$
\end{itemize}
$\blacksquare $ \textsf{The pre-Lie superalgebra $\Big((C_{2_{h},6})_k,\circ \Big)$:} $h=0~~ or~~1,~~k=1$ associative other cases non-associative.
$$ e_1\circ e_1=ke_1,~~e_2\circ e_1=e_2,~~e_3\circ e_1=h e_3,~~h,k\in \mathbb{C}.$$
Rota-Baxter operators $RB((C_{2_{h},6})_k)$ are:\\
$\triangleright$ \textbf{~~Case 1}: If $h=0$, we have
\begin{itemize}
\item[]  $R_{1}(e_1) =0,\ R_{1}(e_2)=0,\ R_{1}(e_3)=a_{1}e_2+a_{2}e_3.$
\item[]  $R_{2}(e_1) = 0,\ R_{2}(e_2)=a_{3}e_3,\ R_{2}(e_3)=a_{2}e_3.$
\item[]  $R_{3}(e_1) =0,\ R_{3}(e_2)= 0,\ R_{3}(e_3)=a_{2}e_3.$
\end{itemize}
$\triangleright$ \textbf{~~Case 2}: If $h \in \mathbb{C}^\ast$, we have
\begin{itemize}
\item[]  $R_{4}(e_1) = 0,\ R_{4}(e_2)=0,\ R_{4}(e_3)=a_{1}e_2.$
\item[]  $R_{5}(e_1) = 0,\ R_{5}(e_2)= a_{3}e_3,\ R_{5}(e_3)=0.$
\item[]  $R_{6}(e_1) =0,\ R_{6}(e_2)=a_{4}e_2+a_{3}e_3,\ R_{6}(e_3)=-\frac{a_{4}^2}{ha_{3}}e_2-\frac{a_{4}}{h}e_3,~~a_{3}\neq 0.$
\item[]  $R_{7}(e_1) = 0,\ R_{7}(e_2)=0,\ R_{7}(e_3)=0.$
\end{itemize}
$\triangleright$ \textbf{~~Case 3}: If $h=k=0$, we have
\begin{itemize}
\item[]  $R_{8}(e_1) = 0,\ R_{8}(e_2)=0,\ R_{8}(e_3)=a_{1}e_2+a_{2}e_3.$
\item[]  $R_{9}(e_1) =0,\ R_{9}(e_2)=a_{3}e_3,\ R_{9}(e_3)=a_{2}e_3.$
\item[]  $R_{10}(e_1) = a_{5}e_1,\ R_{10}(e_2)=0,\ R_{10}(e_3)=a_{2}e_3.$
\item[]  $R_{11}(e_1) = a_{5}e_1,\ R_{11}(e_2)=0,\ R_{11}(e_3)= 0.$
\item[]  $R_{12}(e_1) =0,\ R_{12}(e_2)=a_{3}e_3,\ R_{12}(e_3)=a_{2}e_3,~~a_{3}\neq 0.$
\end{itemize}
$\triangleright$ \textbf{~~Case 4}: If $h=1$ and $k=0$, we have
\begin{itemize}
\item[]  $R_{13}(e_1) = a_{5}e_1,\ R_{13}(e_2)=0,\ R_{13}(e_3)=a_{1}e_2.$
\item[]  $R_{14}(e_1) = a_{5}e_1,\  R_{14}(e_2)=a_{3}e_3,\  R_{14}(e_3)=0.$
\item[]  $R_{15}(e_1) = a_{5}e_1,\  R_{15}(e_2)=a_{4}e_2+a_{3}e_3,\
R_{15}(e_3)=-\frac{a_{4}^2}{a_{3}}e_2-a_{4}e_3,~~a_{3}\neq 0.$
\end{itemize}
$\triangleright$ \textbf{~~Case 5}: If $h \neq 0,1$ and $k=0$, we have
\begin{itemize}
\item[]  $R_{16}(e_1) = 0,\ R_{16}(e_2)=0,\ R_{16}(e_3)=a_{1}e_2,~~a_{1}\neq 0.$
\item[]  $R_{17}(e_1) = 0,\ R_{17}(e_2)=a_{3}e_3,\ R_{17}(e_3)=0,~~r_{2,2}\neq 0.$
\item[]  $R_{18}(e_1) = 0,\ R_{18}(e_2)=a_{4}e_2+a_{3}e_3,\ R_{18}(e_3)= -\frac{a_{4}^2}{ha_{3}}e_2+\frac{a_{4}}{h}e_3,~~a_{3}\neq 0.$
\end{itemize}
$ \blacksquare$ \textsf{The pre-Lie superalgebra $\Big( (C_{2_{h},7})_k,\circ \Big)$:} $(C_{2_{0},7})_1$ associative.
$$ e_1\circ e_1=ke_1,~~e_1\circ e_2=ke_2,~~e_2\circ e_1=(k+1)e_2,~~e_3\circ e_1=h e_3,~~,~~k\neq 0,~~h \neq \pm 1.$$
Rota-Baxter operators $RB((C_{2_{h},7})_k)$ are:\\
$\triangleright$ \textbf{~~Case 1}: If $h=0$, we have
\begin{itemize}
\item[]  $R_{1}(e_1)=0,\ R_{1}(e_2)=0,\ R_{1}(e_3)=a_{1}e_2+a_{2}e_3,~~a_{1}\neq 0.$
\item[] $R_{2}(e_1) = 0,\ R_{2}(e_2)=a_{3}e_3,\ R_{2}(e_3)=a_{2} e_3,~~a_{3}\neq 0.$
\item[]  $R_{3}(e_1)=0,\ R_{3}(e_2)=0,\ R_{3}(e_3)=a_{2}e_3.$
  \end{itemize}
$\triangleright$ \textbf{~~Case 2}: If $h \in \mathbb{C}^\ast$, we have
\begin{itemize}
\item[]  $R_{4}(e_1) =0,\ R_{4}(e_2)=0,\ R_{4}(e_3)=a_{1}e_2,~~a_{1}\neq 0.$
\item[]  $R_{5}(e_1) = 0,\  R_{5}(e_2)=a_{3}e_3,\ R_{5}(e_3)=0,~~a_{3}\neq 0.$
\item[]  $R_{6}(e_1) = 0,\ R_{6}(e_2)=0,\ R_{6}(e_3)= 0.$
  \end{itemize}
$\blacksquare $ \textsf{The pre-Lie superalgebra $\Big( (C_{2_{h},8})_k,\circ \Big)$:} $h=-1~~ or~~0,~~k=1$ associative other cases are non-associative
$$ e_1\circ e_1=ke_1,~~e_1\circ e_3=ke_3,~~e_2\circ e_1=e_2,~~e_3\circ e_1=(h+k) e_3,~~h\in \mathbb{C},~~k\neq 0. $$
Rota-Baxter operators $RB((C_{2_{h},8})_k)$ are:
\begin{itemize}
\item[]  $R_{1}(e_1) =0,\ R_{1}(e_2)=0,\ R_{1}(e_3)=a_{1}e_2.$
\item[] $R_{2}(e_1) = 0,\ R_{2}(e_2)=0,\ R_{2}(e_3)=0.$
  \end{itemize}
$\blacksquare$ \textsf{The pre-Lie superalgebra $\Big((C_{2_{h})_k,9},\circ \Big)$:} $h=0~~ or~~1,~~k=-1$ associative other cases are non-associative
$$ e_1\circ e_1=ke_1,~~e_1\circ e_2=ke_2,~~e_1\circ e_3=ke_3,~~e_2\circ e_1=(k+1)e_2,~~e_3\circ e_1=(h+k) e_3,~~ h\in \mathbb{C},~~k\neq 0. $$
Rota-Baxter operators $RB((C_{2_{h},9})_k)$ are:
\begin{itemize}
\item[]  $R_{1}(e_1) =0,\ R_{1}(e_2)=a_{1}e_2+a_{2}e_3,\ R_{1}(e_3)=-\frac{a_{1}^2}{a_{2}}e_2-a_{1}e_3,~~a_{2}\neq 0,~~h=1,~~k=-1.$
\item[]  $R_{2}(e_1) = 0,\  R_{2}(e_2)=a_{2}e_3,\ R_{2}(e_3)=0.$
\item[]  $R_{3}(e_1) = 0,\ R_{3}(e_2)=0,\ R_{3}(e_3)= a_{3}e_2.$
\item[]  $R_{4}(e_1) = 0,\ R_{4}(e_2)=0,\ R_{4}(e_3)= 0.$
  \end{itemize}
\subsubsection{\textbf{ Rota-Baxter operators on pre-Lie superalgebras of type $C_{3}$}}$ $

$\blacksquare$ \textsf{The pre-Lie superalgebra $(C_{3,1},\circ)$:}
$$ e_1\circ e_1=2e_1,~~e_2\circ e_1=e_2,~~e_2\circ e_3=-e_1,~~e_3\circ e_1=e_2+e_3,~~e_3\circ e_2=e_1. $$
Rota-Baxter operators $RB(C_{3,1})$ are:
\begin{itemize}
\item[]  $R_{1}(e_1)=0,\ R_{1}(e_2)=0,\  R_{1}(e_3)=a_{1}e_2.$
\item[]  $R_{2}(e_1)=0,\  R_{2}(e_2)=a_{2}e_1+a_{3}e_2,\  R_{2}(e_3)=-\frac{a_{2}^2+a_{2}a_{3}}{2}e_2-(a_{2}+a_{3})e_3,~~a_{3}\neq 0.$
\end{itemize}
$\blacksquare $ \textsf{The pre-Lie superalgebra $\Big((C_{3,2})_k,\circ\Big)$:}
$$ e_1\circ e_1=ke_1,~~e_2\circ e_1=e_2,~~e_3\circ e_1=e_2+e_3. $$
Rota-Baxter operators $RB((C_{3,2})_k)$ are:\\
$\triangleright$ \textbf{Case 1}: If $k=0$, we have
\begin{itemize}
\item[]  $R_{1}(e_1) =a_{1}e_1,\ R_{1}(e_2)= 0,\ R_{1}(e_3)=a_{2}e_2,~~a_{1}\neq 0.$
\item[]  $R_{2}(e_1) = a_{1}e_1,\ R_{2}(e_2)=0,\ R_{2}(e_3)=a_{1}e_2,~~a_{1}\neq 0.$
\item[]  $R_{3}(e_1) = a_{1}e_1,\ R_{3}(e_2)=0,\ R_{3}(e_3)= a_{2}e_2,~~a_{1}\neq 0.$
  \end{itemize}
$\triangleright$ \textbf{Case 2}: If $k \in \mathbb{C}^\ast$, we have
\begin{itemize}
\item[]  $R_{4}(e_1) =0,\ R_{4}(e_2)= 0,\ R_{4}(e_3)=a_{2}e_2.$
\item[]  $R_{5}(e_1) = 0,\ R_{5}(e_2)=a_{3}e_2+a_{4}e_3,\ R_{5}(e_3)=-\frac{a_{3}^2+a_{3}a_{4}}{a_{4}}e_2-(a_{3}+a_{4})e_3,~~a_{4}\neq 0.$
\item[]  $R_{6}(e_1) = 0,\ R_{6}(e_2)=0,\ R_{6}(e_3)=0.$
\end{itemize}
$\blacksquare$ \textsf{The pre-Lie superalgebra $\Big((C_{3,3})_k,\circ \Big)$:}
$$ e_1\circ e_3=ke_2,~~e_2\circ e_1=e_2,~~e_3\circ e_1=(k+1)e_2+e_3,~~k\neq 0.$$
Rota-Baxter operators $RB((C_{3,3})_k)$ are:
\begin{itemize}
 \item[] $R_{1}(e_1) = a_{1}e_1,\ R_{1}(e_2)=0,\ R_{1}(e_3)=a_{2}e_2.$
\item[]  $R_{2}(e_1) =0,\ R_{2}(e_2)=0,\ R_{2}(e_3)=a_{2}e_2.$
\item[]  $R_{3}(e_1) =0,\ R_{3}(e_2)=0,\ R_{3}(e_3)=0.$
\end{itemize}
$\blacksquare $ \textsf{The pre-Lie superalgebra $\Big((C_{3,4})_k,\circ\Big)$:}
$$ e_1\circ e_1=ke_1,~~e_1\circ e_2=ke_2,~~e_1\circ e_3=ke_3,~~e_2\circ e_1=(k+1)e_2,~~e_3\circ e_1=e_2+(k+1)e_3,~~k\neq 0. $$
Rota-Baxter operator $RB((C_{3,4})_k)$ is:
\begin{itemize}
 \item[] $R_{1}(e_1) =0,\ R_{1}(e_2)=0,\ R_{1}(e_3)=a_{1}e_2.$
\end{itemize}
\subsubsection{\textbf{Rota-Baxter operators on pre-Lie superalgebras of type $C_{4}$}}$ $

$\blacksquare$ \textsf{The pre-Lie superalgebra $(C_{4,1},\circ)$:} $($associative$)$
$$ ~e_2\circ e_3=-e_1,~~e_3\circ e_2=e_1. $$
Rota-Baxter operators $RB(C_{4,1})$ are:
\begin{itemize}
 \item[] $R_{1}(e_1) =a_{1}e_1,\ R_{1}(e_2)=a_{2}e_2+a_{3}e_3,\ R_{1}(e_3)=a_{4}e_2-\frac{a_{1}a_{2}+a_{3}a_{4}}{a_{1}-a_{2}}e_3,~~a_{1}\neq a_{2}.$
  \item[]$R_{2}(e_1) = 0,\ R_{2}(e_2)=0,\ R_{2}(e_3)=a_{4}e_2+a_{5}e_3.$
  \item[] $R_{3}(e_1) =a_{1}e_1,\ R_{3}(e_2)=a_{1}e_2+a_{3}e_3,\ R_{3}(e_3)=-\frac{a_{1}^2}{a_{3}}e_2+a_{5}e_3,~~a_{3}\neq 0.$
\end{itemize}
$\blacksquare$ \textsf{The pre-Lie superalgebra $(C_{4,2},\circ)$:} (associative).
$$e_i\circ e_j=0,~~\forall~~i,j=1,2,3.$$
Rota-Baxter operator $RB(C_{4,2})$ is:
\begin{itemize}
  \item[]$R_{1}(e_1)=a_{1}e_1,\ R_{1}(e_2)=a_{2}e_2+a_{3}e_3,\ R_{1}(e_3)=a_{4}e_2+a_{5}e_3. $
\end{itemize}
$\blacksquare$ \textsf{The pre-Lie superalgebra $(C_{4,3},\circ)$:} $($associative$)$
$$ e_1\circ e_1=e_1. $$
Rota-Baxter operator $RB(C_{4,3})$ is:
\begin{itemize}
 \item[] $R_{1}(e_1)=0,\ R_{1}(e_2)=a_{1}e_2+a_{2}e_3,\ R_{1}(e_3)=a_{3}e_2+a_{4}e_3.$
\end{itemize}
$\blacksquare$ \textsf{The pre-Lie superalgebra $(C_{4,4},\circ)$:} $($associative$)$
$$ e_1\circ e_1=e_1,~~e_1\circ e_3=e_3,~~e_3\circ e_1=e_3. $$
Rota-Baxter operators $RB(C_{4,4})$ are:
\begin{itemize}
\item[]  $R_{1}(e_1) =0,\ R_{1}(e_2)=a_{1}e_2,\ R_{1}(e_3)=a_{2}e_2.$
\item[]  $R_{2}(e_1) =0,\ R_{2}(e_2)=a_{1}e_2+a_{3}e_3,\ R_{2}(e_3)=0.$
\item[]  $R_{3}(e_1) =0,\ R_{3}(e_2)=a_{1}e_2,\ R_{3}(e_3)=0.$
\end{itemize}
$\blacksquare$ \textsf{The pre-Lie superalgebra $(C_{4,5},\circ)$:} $($associative$)$
$$ e_1\circ e_1=e_1,~~e_1\circ e_2=e_2,~~e_1\circ e_3=e_3,~~e_2\circ e_1=e_2,~~e_3\circ e_1=e_3. $$
Rota-Baxter operators $RB(C_{4,5})$ are:
\begin{itemize}
\item[]  $R_{1}(e_1) = 0,\ R_{1}(e_2)=0,\ R_{1}(e_3)=a_{1}e_2.$
\item[]  $R_{2}(e_1) =0,\ R_{2}(e_2)=a_{2}e_2+a_{3}e_3,\ R_{2}(e_3)=-\frac{a_{2}^2}{a_{3}}e_2-a_{2} e_3,~~a_{3}\neq 0.$
\item[]  $R_{3}(e_1) =0,\ R_{3}(e_2)=0,\ R_{3}(e_3)=0.$
\end{itemize}
$\blacksquare$ \textsf{The pre-Lie superalgebra $(C_{4,6},\circ)$:} $($associative$)$
$$ e_1\circ e_3=e_2,~~e_3\circ e_1=e_2. $$
Rota-Baxter operators $RB(C_{4,6})$ are:
\begin{itemize}
 \item[] $R_{1}(e_1) = 0,\ R_{1}(e_2)=0,\  R_{1}(e_3)=a_{1}e_2+a_{2}e_3.$
 \item[] $R_{2}(e_1) = a_{3}e_1,\ R_{2}(e_2)=a_{4}e_2,\ R_{2}(e_3)=a_{1}e_2+\frac{a_{3}a_{4}}{a_{3}-a_{4}} e_3,~~a_{3}\neq a_{4}.$
\end{itemize}
\subsubsection{\textbf{Rota-Baxter operators on pre-Lie superalgebras of type $C_{5}$}}$ $

$\blacksquare$ \textsf{The pre-Lie superalgebra $\Big((C_{5,1})_k,\circ\Big)$:} $((C_{5,1})_0$ is associative$)$.
$$ e_1\circ e_2=ke_3,~~e_2\circ e_1=ke_3,~~e_3\circ e_3=e_1,~~k=0~~or~~1. $$
Rota-Baxter operators $RB((C_{5,1})_k)$ are:\\
$\triangleright$ \textbf{Case 1}: If $k=0$, we have
\begin{itemize}
\item[]  $R_{1}(e_1) =0,\ R_{1}(e_2)=a_{1}e_3,\ R_{1}(e_3)=a_{2}e_3.$
\item[]  $R_{2}(e_1) = 0,\ R_{2}(e_2)=a_{1}e_3,\  R_{2}(e_3)=a_{3}e_2.$
\item[]  $R_{3}(e_1) = 0,\ R_{3}(e_2)=a_{4}e_2,\  R_{3}(e_3)=a_{3}e_2.$
\item[]  $R_{4}(e_1) = 0,\ R_{4}(e_2)=0,\  R_{4}(e_3)=a_{2}e_3.$
\item[]  $R_{5}(e_1) =a_{5}e_1 ,\ R_{5}(e_2)=a_{4}e_2,\  R_{5}(e_3)=\frac{a_{5}a_{4}}{a_{5}-a_{4}}e_3,~~a_{4}\neq a_{5}.$
\end{itemize}
$\triangleright$ \textbf{Case 2}: If $k=1$, we have
\begin{itemize}
\item[]  $R_{6}(e_1) =0,\ R_{6}(e_2)=a_{1}e_3,\  R_{6}(e_3)=a_{2}e_3.$
\item[]  $R_{7}(e_1) =0,\  R_{7}(e_2)=0,\ R_{7}(e_3)=a_{2}e_3.$
\item[]  $R_{8}(e_1) =0,\  R_{8}(e_2)=a_{1}e_3,\ R_{8}(e_3)=e_3.$
\item[]  $R_{9}(e_1) =a_{5}e_1,\  R_{9}(e_2)=a_{5}e_2,\ R_{9}(e_3)=\frac{a_{5}}{2}e_3,~~a_{5}\neq 0.$
\item[]  $R_{10}(e_1) =a_{5}e_1,\  R_{10}(e_2)=\frac{a_{5}}{a_{5}-1}e_2,\ R_{10}(e_3)=e_3,~~a_{5}\neq 1.$
\end{itemize}
$\blacksquare$ \textsf{The pre-Lie superalgebra $(C_{5,2},\circ)$:}
$$ e_1\circ e_1=e_1,~~e_1\circ e_2=e_2,~~e_2\circ e_1=e_2,~~e_2\circ e_3=e_1. $$
Rota-Baxter operators $RB(C_{5,2})$ are:
\begin{itemize}
\item[]  $R_{1}(e_1) =0,\  R_{1}(e_2)=a_{1}e_3,\ R_{1}(e_3)=a_{2}e_3.$
\item[]  $R_{2}(e_1)=0,\ R_{2}(e_2)=0,\  R_{2}(e_3)=a_{2}e_3.$
\item[] $R_{3}(e_1) =0,\ R_{3}(e_2)=0,\  R_{3}(e_3)=a_{3}e_2.$
\end{itemize}
$\blacksquare$ \textsf{The pre-Lie superalgebra $(C_{5,3},\circ)$:}
$$ e_1\circ e_1=e_1,~~e_1\circ e_2=e_2+e_3,~~e_2\circ e_1=e_2+e_3,~~e_2\circ e_3=e_1. $$
Rota-Baxter operators $RB(C_{5,3})$ are:
\begin{itemize}
\item[]  $R_{1}(e_1) =0,\ R_{1}(e_2)=a_{1}e_3,\  R_{1}(e_3)=a_{2}e_3.$
\item[]  $R_{2}(e_1) =0,\  R_{2}(e_2)=0,\ R_{2}(e_3)=a_{2}e_3.$
\item[]  $R_{3}(e_1) =0,\ R_{3}(e_2)=a_{3}e_2,\  R_{3}(e_3)=-a_{3}e_2.$
\end{itemize}
$\blacksquare$ \textsf{The pre-Lie superalgebra $\Big((C_{5,4})_k,\circ\Big)$:} $($associative$)$
$$ e_2\circ e_3=(\frac{1}{2}+k)e_1,~~e_3\circ e_2=(\frac{1}{2}-k)e_1,~~ k\geq 0,~~k\neq \frac{1}{2}. $$
Rota-Baxter operators $RB((C_{5,4})_k)$ are:
\begin{itemize}
\item[]  $R_{1}(e_1) =0,\  R_{1}(e_2) =a_{1}e_3,\  R_{1}(e_3) =a_{2}e_3.$
 \item[] $R_{2}(e_1)= 0,\  R_{2}(e_2)=a_{3}e_2,\  R_{2}(e_3)=a_{4}e_2.$
\item[]  $R_{3}(e_1)=0,\ R_{3}(e_2)=0,\  R_{3}(e_3)=a_{2}e_3.$
\item[]  $R_{4}(e_1)=0,\  R_{4}(e_2) =0,\   R_{4}(e_3) =a_{4} e_2.$
\item[] $R_{5}(e_1)=0,\ R_{5}(e_2)=a_{1}e_3,\  R_{5}(e_3)=0,~~a_{1}\neq 0.$
\item[] $R_{6}(e_1)=0,\  R_{6}(e_2)=a_{3}e_2,\ R_{6}(e_3)=0,~~a_{3}\neq 0.$
\item[]  $R_{7}(e_1)=a_{5}e_1,\  R_{7}(e_2) =0,\   R_{7}(e_3) =0.$
\item[] $R_{8}(e_1)=a_{5}e_1,\ R_{8}(e_2)=a_{5}e_2+a_{1}e_3,\  R_{8}(e_3)=\frac{a_{5}^2}{a_{1}}e_2+a_{5}e_3,~~a_{1}\neq 0,~~a_{5}\neq 0,~~k=0.$
\item[] $R_{9}(e_1)=a_{5}e_1,\  R_{9}(e_2)=a_{3}e_2,\ R_{9}(e_3)=-\frac{a_{5}a_{3}}{a_{5}-a_{3}}e_3,~~a_{3}\neq a_{5}.$
\end{itemize}
\subsubsection{\textbf{Rota-Baxter operators on pre-Lie superalgebras of type $C_{6}$}}$ $

$\blacksquare$ \textsf{The pre-Lie superalgebra $(C_{6,1},\circ)$} (associative)
$$ e_2\circ e_2=\frac{1}{2}e_1,~~e_2\circ e_3=-e_1,~~e_3\circ e_2=e_1. $$
Rota-Baxter operators $RB(C_{6,1})$ are:
\begin{itemize}
\item[]  $R_{1}(e_1) =0,\  R_{1}(e_2)=a_{1}e_3,\ R_{1}(e_3)=a_{2}e_3,~~a_{1}\neq 0.$
\item[]  $R_{2}(e_1) =a_{3}e_1,\ R_{2}(e_2)=a_{1}e_3,\  R_{2}(e_3)=0.$
\end{itemize}

$\blacksquare$ \textsf{The pre-Lie superalgebra $(C_{6,2},\circ)$:} $($associative$)$
$$ e_2\circ e_2=\frac{1}{2}e_1. $$
Rota-Baxter operators $RB(C_{6,2})$ are:
\begin{itemize}
\item[]  $R_{1}(e_1) =a_{1}e_1,\  R_{1}(e_2)=a_{2}e_3,\ R_{1}(e_3)=a_{3}e_3.$
\item[]  $R_{2}(e_1) = a_{1}e_1,\ R_{2}(e_2)=2a_{1}e_2+a_{2}e_3,\  R_{2}(e_3)=a_{3}e_3.$
\end{itemize}
$\blacksquare$ \textsf{The pre-Lie superalgebra $(C_{6,3},\circ)$} (associative)
$$ e_1\circ e_2=e_3,~~e_2\circ e_1=e_3,~~e_2\circ e_2=\frac{1}{2}e_1. $$
Rota-Baxter operators $RB(C_{6,3})$ are:
\begin{itemize}
\item[]  $R_{1}(e_1)=0,\  R_{1}(e_2)=a_{1}e_2,\  R_{1}(e_3)=a_{2}e_3.$
\item[]  $R_{2}(e_1)=a_{3}e_1,\  R_{2}(e_2)= a_{1}e_3,\  R_{2}(e_3)=0,~~a_{3}\neq 0.$
\item[]  $R_{3}(e_1)=a_{3}e_1,\  R_{3}(e_2)= 2a_{3}e_2+a_{1}e_3,\  R_{3}(e_3)=\frac{2a_{3}}{3}e_3.$
\end{itemize}
$\blacksquare$ \textsf{The pre-Lie superalgebra $(C_{6,4},\circ)$:} $($associative$)$
$$ e_1\circ e_1=e_1,~~e_1\circ e_2=e_2,~~e_2\circ e_1= e_2,~~~e_2\circ e_2=\frac{1}{2}e_1. $$
Rota-Baxter operators $RB(C_{6,4})$ are:
\begin{itemize}
 \item[] $R_{1}(e_1) =0,\ R_{1}(e_2)=a_{1}e_3,\ R_{1}(e_3)=a_{2}e_3.$
\item[]  $R_{2}(e_1) =0,\ R_{2}(e_2)= 0,\ R_{2}(e_3)=a_{2}e_3.$
\end{itemize}
\subsection{Classification of Rota-Baxter operator on $3$-dimensional pre-Lie superalgebras with $2$-dimensional even part}
In this section, we describe all   Rota-Baxter operators of weight zero on the $3$-dimensional complex pre-Lie superalgebras with two-dimensional  even part which were classified in \cite{Bai and Zhang classif} by R. Zhang and C. Bai.
In the following, let $\{e_1,e_2,e_3\}$ be a homogeneous basis of a pre-Lie superalgebra $(\mathcal{A},\circ)$, where  $\{e_1,e_2\}$ is a basis of $\mathcal{A}_0$ and $\{e_3\}$ is a basis of $\mathcal{A}_1$.
The computation are obtained using computer algebra system and the operators are described with respect to the basis.
\begin{prop}The Rota-Baxter operators $($of weight zero$)$ on $3$-dimensional pre-Lie superalgebras (associative or non-associative) with $2$-dimensional even part of type $\widehat{A}_1,~~\widehat{A}_{2},~~\widehat{A}_{3},~~\widehat{A}_{4},~~\widehat{A}_{5},~~\widehat{A}_6,~~\widehat{A}_{7_h},
~~\widehat{A}_8,~~\widehat{A}_9,~~\widehat{A}_{10_h}$ and $\widehat{A}_{11}$ are given as follows:
\end{prop}
\subsubsection{\textbf{Rota-Baxter operators on pre-Lie superalgebras of type} $\widehat{A}_{1}$}$ $

$\blacksquare$ \textsf{The pre-Lie superalgebra $(\widehat{A}_{1,1},\circ)\simeq D_3$} (associative).
$$ e_1\circ e_1=e_1,~~e_2\circ e_2=e_2. $$
Rota-Baxter operators $RB(\widehat{A}_{1,1})$ are:
\begin{itemize}
\item[]  $R_{1}(e_1) =0,\ R_{1}(e_2)=0,\ R_{1}(e_3)=a_{1}e_3.$
\item[]  $R_{2}(e_1) =0,\ R_{2}(e_2)= \frac{1}{2}e_2,\  R_{2}(e_3)=a_{1}e_3.$
\item[]  $R_{3}(e_1) =0,\ R_{3}(e_2)= e_1+\frac{1}{2}e_2,\  R_{3}(e_3)=a_{1}e_3.$
\end{itemize}
$\blacksquare$ \textsf{The pre-Lie superalgebra $\Big((\widehat{A}_{1,2})_k,\circ\Big)\simeq D_1$}, $(\widehat{A}_{1,2})_1$ associative.
$$e_1\circ e_1=e_1,~~e_2\circ e_2=e_2,~~~e_3\circ e_1=ke_3,~~k\in \mathbb{C}^\ast. $$
Rota-Baxter operators $RB((\widehat{A}_{1,2})_k)$ are:
\begin{itemize}
\item[]  $R_{1}(e_1) =0,\ R_{1}(e_2)=0,\  R_{1}(e_3)=0.$
\item[]  $R_{2}(e_1) =0,\  R_{2}(e_2)=\frac{1}{2}e_2,\  R_{2}(e_3)=0.$
\item[]  $R_{3}(e_1) =0,\  R_{3}(e_2)=e_1+\frac{1}{2}e_2,\  R_{3}(e_3)=0.$
\end{itemize}
$\blacksquare$ \textsf{The pre-Lie superalgebra $\Big((\widehat{A}_{1,3})_{k_1,k_2},\circ \Big)\simeq D_1$}.
$$ e_1\circ e_1=e_1,~~e_2\circ e_2=e_2,~~e_3\circ e_1=k_1e_3,~~e_3\circ e_2=k_2e_3,~~k_1, k_2\neq 0,~~k_1\leq k_2,~~k_1\neq -k_2. $$
Rota-Baxter operators $RB((\widehat{A}_{1,3})_{k_1,k_2})$ are:\\
 If $\Big(k_2< 0,~~~~k_1\leq k_2 \Big)$ or $\Big((k_2>0,~~~~k_1<-k_2)~~or~~(-k_2<k_1<0)~~or~~(0<k_1\leq k_2)\Big)$
\begin{itemize}
 \item[] $R_{1}(e_1)=0,\ R_{1}(e_2)=0,\ R_{1}(e_3)=0.$
\end{itemize}
$\blacksquare$ \textsf{The pre-Lie superalgebra $(\widehat{A}_{1,4},\circ)\simeq D_3$} (associative).
$$e_1\circ e_1=e_1,~~e_2\circ e_2=e_2,~~~e_2\circ e_3=e_3,~~e_3 \circ e_2=e_3. $$
Rota-Baxter operators $RB(\widehat{A}_{1,4})$ are:
\begin{itemize}
\item[]  $R_{1}(e_1) =0,\ R_{1}(e_2)=0, \ R_{1}(e_3)=0.$
\item[]  $R_{2}(e_1) =0,\ R_{2}(e_2)=\frac{1}{2}e_2,\ R_{2}(e_3)=0.$
\item[]  $R_{3}(e_1) =0,\ R_{3}(e_2)=e_1+\frac{1}{2}e_2,\ R_{3}(e_3)=0.$
\end{itemize}
$\blacksquare$ \textsf{The pre-Lie superalgebra $\Big((\widehat{A}_{1,5})_{k_1,k_2},\circ\Big)\simeq D_1$}, $k_1=k_2=0~~or~~k_1=1,~~k_2=0$ associative other cases non-associative.
$$e_1\circ e_1=e_1,~~e_2\circ e_2=e_2,~~~e_2\circ e_3=e_3,~~e_3\circ e_1=k_1e_3,~~e_3\circ e_2=k_2e_3,~~k_1\neq 0~~or~~k_2\neq 1. $$
Rota-Baxter operator $RB((\widehat{A}_{1,5})_{k_1,k_2})$ is:
\begin{itemize}
\item[]  $R_{1}(e_1) =0,\ R_{1}(e_2)=0,\ R_{1}(e_3)=0.$
\end{itemize}
$\blacksquare$ \textsf{The pre-Lie superalgebra $(\widehat{A}_{1,6},\circ)\simeq D_2$} (associative).
$$ e_1\circ e_1=e_1,~~e_2\circ e_2=e_2,~~~,e_2\circ e_3=e_3,~~e_3\circ e_2=e_3,~~e_3\circ e_3=e_2. $$
Rota-Baxter operators $RB(\widehat{A}_{1,6})$ are:
\begin{itemize}
\item[]  $R_{1}(e_1) =0,\ R_{1}(e_2)=0, \ R_{1}(e_3)=0.$
\item[]  $R_{2}(e_1) =0,\ R_{2}(e_2)=\frac{1}{2}e_2,\ R_{2}(e_3)=0.$
\item[]  $R_{3}(e_1) =0,\ R_{3}(e_2)=e_1+\frac{1}{2}e_2,\ R_{3}(e_3)=0.$
\end{itemize}
\subsubsection{\textbf{Rota-Baxter operators on pre-Lie superalgebras of type} $\widehat{A}_{2}$}$ $

$\blacksquare$ \textsf{The pre-Lie superalgebra $(\widehat{A}_{2,1},\circ)\simeq D_3$} (associative).
$$ e_1\circ e_1=e_1,~~e_1\circ e_2=e_2,~~~e_2\circ e_1=e_2. $$
Rota-Baxter operators $RB(\widehat{A}_{2,1})$ are:
\begin{itemize}
\item[]  $R_{1}(e_1) =a_{1}e_2,\ R_{1}(e_2)=0,\ R_{1}(e_3)=a_{2}e_3.$
\item[]  $R_{2}(e_1) = 0,\ R_{2}(e_2)=\frac{1}{2}e_1,\ R_{2}(e_3)= a_{2}e_3.$
  \end{itemize}
$\blacksquare$ \textsf{The pre-Lie superalgebra $\Big((\widehat{A}_{2,2})_k,\circ\Big)\simeq D_1$}, $(\widehat{A}_{2,2})_1$ is associative.
$$ e_1\circ e_1=e_1,~~e_1\circ e_2=e_2,~~~e_2\circ e_1=e_1,~~e_3\circ e_1=ke_3,~~k\neq 0. $$
Rota-Baxter operators $RB((\widehat{A}_{2,2})_k)$ are:
\begin{itemize}
\item[]  $R_{1}(e_1) =a_{1}e_2,\ R_{1}(e_2)=0,\ R_{1}(e_3)=a_{2}e_3.$
\item[]  $R_{2}(e_1) = 0,\ R_{2}(e_2)=\frac{1}{2}e_1,\ R_{2}(e_3)= a_{2}e_3.$
  \end{itemize}
$\blacksquare$ \textsf{The pre-Lie superalgebra $\Big((\widehat{A}_{2,3})_k,\circ\big)\simeq D_1$}
$$ e_1\circ e_1=e_1,~~e_1\circ e_2=e_2,~~~e_2\circ e_1=e_2,~~e_3\circ e_1=ke_3,~~e_3\circ e_2=e_3. $$
Rota-Baxter operators $RB((\widehat{A}_{2,3})_k)$ are:
\begin{itemize}
\item[]  $R_{1}(e_1) =a_{1}e_2,\ R_{1}(e_2)=0,\ R_{1}(e_3)=a_{2}e_3.$
\item[]  $R_{2}(e_1) = 0,\ R_{2}(e_2)=\frac{1}{2}e_1,\ R_{2}(e_3)= 0.$
\item[]  $R_{3}(e_1) =a_{1}e_2,\ R_{3}(e_2)=0,\ R_{3}(e_3)=a_{2}e_3,~~k=1.$
\item[]  $R_{4}(e_1) =0,\ R_{4}(e_2)=0,\ R_{4}(e_3)=0,~~k\neq 1.$
  \end{itemize}
$\blacksquare$ \textsf{The pre-Lie superalgebra $(\widehat{A}_{2,4},\circ)\simeq D_3$} (associative).
$$ e_1\circ e_1=e_1,~~e_1\circ e_2=e_2,~~e_1\circ e_3=e_3,~~e_2\circ e_1=e_2,~~~e_3\circ e_1=e_3. $$
Rota-Baxter operator $RB(\widehat{A}_{2,4})$ is:
\begin{itemize}
\item[]  $R_{1}(e_1) =a_{1}e_2,\ R_{1}(e_2)=0,\ R_{1}(e_3)=0.$
  \end{itemize}
$\blacksquare$ \textsf{The pre-Lie superalgebra $\Big((\widehat{A}_{2,5})_k,\circ\Big)\simeq D_1$}, $(\widehat{A}_{2,5})_0$ is associative.
$$ e_1\circ e_1=e_1,~~e_1\circ e_2=e_2,~~e_1\circ e_3=e_3,~~e_2\circ e_1=e_2,~~~e_3\circ e_1=ke_3,~~k\neq 1. $$
Rota-Baxter operators $RB((\widehat{A}_{2,5})_k)$ are:
\begin{itemize}
 \item[] $R_{1}(e_1) = a_{1}e_2,\ R_{1}(e_2)=0,\ R_{1}(e_3)= a_{2}e_3,~~k=0.$
\item[]  $R_{2}(e_1) = a_{1} e_2,\ R_{2}(e_2)=0,\  R_{2}(e_3)=0,~~k\neq 0.$
  \end{itemize}
$\blacksquare$ \textsf{The pre-Lie superalgebra $\Big((\widehat{A}_{2,6})_k,\circ\Big)\simeq D_1$:}
$$ e_1\circ e_1=e_1,~~e_1\circ e_2=e_2,~~e_1\circ e_3=e_3,~~e_2\circ e_1=e_2,~~~e_3\circ e_1=ke_3,~~e_3\circ e_2=e_3. $$
Rota-Baxter operators $RB((\widehat{A}_{2,6})_k)$ are:
\begin{itemize}
\item[]  $R_{1}(e_1) =a_{1}e_2,\ R_{1}(e_2)=0,\ R_{1}(e_3)=0.$
  \end{itemize}
$\blacksquare$ \textsf{The pre-Lie superalgebra $(\widehat{A}_{2,7},\circ)\simeq D_2$} (associative).
$$ e_1\circ e_1=e_1,~~e_1\circ e_2=e_2,~~e_1\circ e_3=e_3,~~e_2\circ e_1=e_2,~~~e_3\circ e_1=e_3,~~e_3\circ e_3=e_2. $$
Rota-Baxter operators $RB(\widehat{A}_{2,7})$ are:
\begin{itemize}
\item[]  $R_{1}(e_1) =a_{1}e_2,\ R_{1}(e_2)=0,\ R_{1}(e_3)=0.$
  \end{itemize}
\subsubsection{\textbf{Rota-Baxter operators on pre-Lie superalgebras of type} $\widehat{A}_{3}$} $ $

$\blacksquare$ \textsf{The pre-Lie superalgebra $(\widehat{A}_{3,1},\circ)\simeq D_3$} (associative).
$$ e_1\circ e_1=e_1. $$
Rota-Baxter operators $RB(\widehat{A}_{3,1})$ are:
\begin{itemize}
 \item[] $R_{1}(e_1) = a_{1}e_2,\ R_{1}(e_2)= a_{2}e_2,\ R_{1}(e_3)=a_{3} e_3.$
  \end{itemize}
$\blacksquare$ \textsf{The pre-Lie superalgebra $(\widehat{A}_{3,4},\circ)\simeq D_2 $:} (associative).
$$ e_1\circ e_1=e_1,~~~e_3\circ e_3=e_2. $$
Rota-Baxter operators $RB(\widehat{A}_{3,4})$ are
\begin{itemize}
\item[]  $R_{1}(e_1) = a_{1}e_2,\ R_{1}(e_2)=a_{2}e_2,\ R_{1}(e_3)=0.$
\item[]  $R_{2}(e_1) = a_{1}e_2,\ R_{2}(e_2)= a_{2}e_2,\ R_{2}(e_3)=2 a_{2}e_3.$
  \end{itemize}
$\blacksquare$ \textsf{The pre-Lie superalgebras} of type
\begin{itemize}
\item[] $\Big((\widehat{A}_{3,2})_k,\circ\Big)\simeq D_1: \ \ e_1\circ e_1=e_1,~~e_3\circ e_1=k e_3,~~k\neq 0,~~((\widehat{A}_{3,2})_1~~ \text{is associative)}.$
\item[] $(\widehat{A}_{3,3})_k,\circ)\simeq D_1 \ \ \ : \ \ \ e_1\circ e_1=e_1,~~e_3\circ e_1=ke_3,~~e_3\circ e_2=e_3. $
\item[] $(\widehat{A}_{3,5},\circ)\simeq D_3 \ \ \ \ \  : \ \ \ e_1\circ e_1=e_1,~~~e_1\circ e_3=e_3,~~e_3\circ e_1=e_3,~~\text{(associative)}.$
\item[] $\Big((\widehat{A}_{3,6})_k,\circ\Big)\simeq D_1 \ ~~  : \ \ \ e_1\circ e_1=e_1,~~~e_1\circ e_3=e_3,~~e_3\circ e_1=ke_3,~~k\neq 1,~~((\widehat{A}_{3,6})_0~~ \text{is associative)}.$
\item[] $\Big((\widehat{A}_{3,7})_k,\circ\Big)\simeq D_1 \ : \ \ \ \ e_1\circ e_1=e_1,~~~e_1\circ e_3=e_3,~~~e_3\circ e_1=ke_3,~~e_3\circ e_2=e_3.$
\item[] $(\widehat{A}_{3,8},\circ)\simeq D_2 \ \ \ \ \ : \ \ \ \ e_1\circ e_1=e_1,~~e_1\circ e_3=e_3,~~~e_3\circ e_1=e_3,~~e_3\circ e_3=e_1,~~\text{(associative)}.$
\end{itemize}
have the same Rota-Baxter operators:
\begin{itemize}
\item[]  $R_{1}(e_1) = a_{1}e_2,\ R_{1}(e_2)=a_{2}e_2,\ R_{1}(e_3)=0.$
  \end{itemize}
\subsubsection{\textbf{ Rota-Baxter operators on pre-Lie superalgebras of type} $\widehat{A}_{4}$} $ $

$\blacksquare$ \textsf{The pre-Lie superalgebra $(\widehat{A}_{4,1},\circ)\simeq D_3$} (associative).
$$e_i\circ e_j=0,~~\forall~~i,j=1,2,3.$$
Rota-Baxter operators $RB(\widehat{A}_{4,1})$ are:
\begin{eqnarray*}
  R_{1}(e_1) &=&a_{1}e_1+a_{2}e_2,\ R_{1}(e_2)=a_{3}e_1+a_{4}e_2,\  R_{1}(e_3)=a_{5} e_3.
  \end{eqnarray*}
$\blacksquare$ \textsf{The pre-Lie superalgebra $(\widehat{A}_{4,2},\circ)\simeq D_1$:}
$$ e_3\circ e_1=e_3. $$
Rota-Baxter operators $RB(\widehat{A}_{4,2})$ are:
\begin{itemize}
\item[]  $R_{1}(e_1) = a_{1}e_1+a_{2}e_2,\ R_{1}(e_2)= a_{3}e_1+a_{4}e_2,\ R_{1}(e_3)=0.$
  \end{itemize}
$\blacksquare$ \textsf{The pre-Lie superalgebra $(\widehat{A}_{4,3},\circ)\simeq D_2$:} $($associative$)$
$$ e_3\circ e_3=e_2. $$
Rota-Baxter operators $RB(\widehat{A}_{4,3})$ are:
\begin{itemize}
\item[]  $R_{1}(e_1) = a_{1}e_1+a_{2}e_2,\ R_{1}(e_2)= a_{3}e_2,\  R_{1}(e_3)= 2a_{3}e_3.$
\item[]  $R_{2}(e_1) = a_{1}e_1+a_{2}e_2,\ R_{2}(e_2)=a_{4}e_1+ a_{3}e_2,\ R_{2}(e_3)=0.$
  \end{itemize}
\subsubsection{\textbf{Rota-Baxter operators on pre-Lie superalgebras of type }$\widehat{A}_{5}$}$ $

$\blacksquare$ \textsf{The pre-Lie superalgebra $(\widehat{A}_{5,1},\circ)\simeq D_2$} (associative)
$$ e_1\circ e_1=e_2,~~~e_3\circ e_3= e_2. $$
Rota-Baxter operators $RB(\widehat{A}_{5,1})$ are:
\begin{itemize}
\item[]  $R_{1}(e_1) = a_{1}e_2,\  R_{1}(e_2)= a_{2}e_2,\ R_{1}(e_3)= 0.$
\item[]  $R_{2}(e_1) = a_{1}e_2,\ R_{2}(e_2)= a_{2}e_2,\ R_{2}(e_3)=2a_{2}e_3.$
\item[]  $R_{3}(e_1) = a_{3}e_1+a_{1}e_2,\  R_{3}(e_2)= \frac{a_{3}}{2}e_2,\ R_{3}(e_3)= 0.$
\item[]  $R_{4}(e_1) = a_{3}e_1+a_{1}e_2,\ R_{4}(e_2)= \frac{a_{3}}{2}e_2,\ R_{4}(e_3)=a_{3}e_3.$
  \end{itemize}
$\blacksquare$ \textsf{The pre-Lie superalgebra $(\widehat{A}_{5,2},\circ)\simeq D_3$} (associative).
$$ e_1\circ e_1=e_2. $$
Rota-Baxter operators $RB(\widehat{A}_{5,2})$ are:
\begin{itemize}
\item[] $R_{1}(e_1) = a_{1}e_1+a_{2}e_2,\ R_{1}(e_2)= \frac{a_{1}}{2}e_2,\ R_{1}(e_3)=a_{3}e_3.$
\item[]  $R_{2}(e_1) =a_{2}e_2,\ R_{2}(e_2)=a_{4}e_2,\ R_{2}(e_3)= 0.$
  \end{itemize}
$\blacksquare$ \textsf{The pre-Lie superalgebra $(\widehat{A}_{5,3},\circ)\simeq D_1$:}
$$ e_1\circ e_1=e_2,~~e_3\circ e_1=e_3. $$
Rota-Baxter operators $RB(\widehat{A}_{5,3})$ are:
\begin{itemize}
\item[]  $R_{1}(e_1) = a_{1}e_1+a_{2}e_2,\  R_{1}(e_2)= \frac{a_{1}}{2}e_2,\ R_{1}(e_3)=0.$
 \item[] $R_{2}(e_1) =a_{2}e_2,\ R_{2}(e_2)=a_{3}e_2,\ R_{2}(e_3)= 0.$
\end{itemize}
$\blacksquare$ \textsf{The pre-Lie superalgebra $(\widehat{A}_{5,4},\circ)\simeq D_1$:}
$$ e_1\circ e_1=e_2~,~~~e_3\circ e_2=e_3. $$
Rota-Baxter operators $RB(\widehat{A}_{5,4})$ are:
\begin{itemize}
\item[]  $R_{1}(e_1) = a_{1}e_1+a_{2}e_2,\  R_{1}(e_2)= \frac{a_{1}}{2}e_2,\ R_{1}(e_3)=0.$
 \item[] $R_{2}(e_1) =a_{2}e_2,\ R_{2}(e_2)=a_{3}e_2,\ R_{2}(e_3)= 0.$
  \end{itemize}
\subsubsection{\textbf{Rota-Baxter operators on pre-Lie superalgebras of type }$\widehat{A}_{6}$}$ $

$\blacksquare$ \textsf{The pre-Lie superalgebra $\Big((\widehat{A}_{6,1})_k,\circ\Big)\simeq (D_4)_\mu: k=0~~or~~-1$ associative other cases non-associative}
$$ ~e_1\circ e_2=-e_1,~~e_2\circ e_2=-e_2,~~~e_3\circ e_2=ke_3. $$
Rota-Baxter operators $RB((\widehat{A}_{6,1})_k)$ are:\\
$\triangleright$ \textbf{Case 1}: If $k=0$, we have
\begin{itemize}
\item[]  $R_{1}(e_1) = a_{1}e_2,\ R_{1}(e_2)= 0,\  R_{1}(e_3)=a_{2}e_3.$
\item[]  $R_{2}(e_1) = 0,\ R_{2}(e_2)=a_{3}e_1,\  R_{2}(e_3)=a_{2} e_3.$
  \end{itemize}
$\triangleright$ \textbf{Case 2}: If $k \in \mathbb{C}^\ast$, we have
\begin{itemize}
 \item[] $R_{4}(e_1) =0,\  R_{4}(e_2)=a_{3}e_1,\  R_{4}(e_3)=0.$
 \item[] $R_{5}(e_1) = a_{1}e_2,\  R_{5}(e_2)=0,\  R_{5}(e_3)=0.$
 \item[] $R_{5}(e_1) = a_{1}e_2,\  R_{5}(e_2)=0,\  R_{5}(e_3)=0,~~ a_{1}\neq 0.$
  \end{itemize}
$\blacksquare$ \textsf{The pre-Lie superalgebra $(\widehat{A}_{6,2},\circ)\simeq D_5$}:
$$ ~e_1\circ e_2=-e_1,~~e_2\circ e_2=-e_2,~~~e_3\circ e_2=-\frac{1}{2}e_3,~~e_3\circ e_3=e_1. $$
Rota-Baxter operators $RB(\widehat{A}_{6,2})$ are:
\begin{itemize}
\item[]  $R_{1}(e_1) = a_{1}e_2,\ R_{1}(e_2)=0,\ R_{1}(e_3)=0.$
 \item[] $R_{2}(e_1) =0,\ R_{2}(e_2)=a_{2}e_1,\ R_{2}(e_3)=0.$
  \end{itemize}
$\blacksquare$ \textsf{The pre-Lie superalgebra $\Big((\widehat{A}_{6,3})_k,\circ\Big)\simeq (D_4)_\mu:~~k=0~~or~~-1$ associative other cases are non-associative}
$$ ~e_1\circ e_2=-e_1,~~e_2\circ e_2=-e_2,~~~e_2\circ e_3=-e_3,~~e_3\circ e_2=ke_3. $$
Rota-Baxter operators $RB((\widehat{A}_{6,3})_k)$ are:
\begin{itemize}
\item[]  $R_{1}(e_1) = a_{1}e_2,\ R_{1}(e_2)=0,\ R_{1}(e_3)=0,~~a_{1}\neq 0.$
\item[]  $R_{2}(e_1) =0,\ R_{2}(e_2)=a_{2}e_1,\ R_{2}(e_3)=0.$
  \end{itemize}
\subsubsection{\textbf{Rota-Baxter operators on pre-Lie superalgebras of type }$\widehat{A}_{7_{h}}$}$ $

$\blacksquare$ \textsf{The pre-Lie superalgebra $(\widehat{A}_{7_{h},1},\circ)\simeq D_5$:}
$$ ~e_1\circ e_2=-e_1,~~e_2\circ e_2=h e_2,~~~e_3\circ e_2=-\frac{1}{2}e_3,~~e_3\circ e_3=e_1,~~h\neq -1. $$
Rota-Baxter operators $RB(\widehat{A}_{7_{h},1})$ are:
\begin{itemize}
\item[]  $R_{1}(e_1) = 0,\ R_{1}(e_2)=a_1e_1,\ R_{1}(e_3)=0.$
\item[]  $R_{2}(e_1) = 0,\ R_{2}(e_2)=a_2e_2,\ R_{2}(e_3)=0,~~h=0.$
\item[]  $R_{3}(e_1) =0,\ R_{3}(e_2)=0,\ R_{3}(e_3)=0,~~h\neq 0.$
  \end{itemize}
$\blacksquare$ \textsf{The pre-Lie superalgebra $\Big((\widehat{A}_{7_{h},2})_k,\circ\Big)\simeq (D_4)_\mu$:}
$$ ~e_1\circ e_2=-e_1,~~e_2\circ e_2=h e_2,~~~e_2\circ e_3=h e_3,~~e_3\circ e_2=ke_3,~~h\neq 0,-1. $$
Rota-Baxter operators $RB((\widehat{A}_{7_{h},2})_k)$ are:
\begin{itemize}
\item[]  $R_{1}(e_1) = 0,\ R_{1}(e_2)=a_1e_1,\ R_{1}(e_3)=0.$
\item[]  $R_{2}(e_1) =0,\ R_{2}(e_2)=0,\ R_{2}(e_3)=0.$
  \end{itemize}
$\blacksquare$ \textsf{The pre-Lie superalgebra $\Big((\widehat{A}_{7_{h},3})_k,\circ\Big)\simeq (D_4)_\mu$:}
$$ ~e_1\circ e_2=-e_1,~~e_2\circ e_2=h e_2,~~~e_3\circ e_2=ke_3,~~h\neq -1. $$
Rota-Baxter operators $RB((\widehat{A}_{7_{h},3})_k)$ are:
\begin{itemize}
\item[]  $R_{1}(e_1) = a_{1}e_2,\ R_{1}(e_2)=a_{2}e_1,\ R_{1}(e_3)=a_{3}e_3,~~h=-\frac{1}{2},~~k=0.$
 \item[] $R_{2}(e_1) =0,\ R_{2}(e_2)=a_{4}e_2,\ R_{2}(e_3)=0,~~h=0.$
 \item[] $R_{3}(e_1) = 0,\ R_{3}(e_2)=a_{4}e_2,\ R_{3}(e_3)=a_{3}e_3,~~h=k=0.$
 \item[] $R_{4}(e_1) = 0, \ R_{4}(e_2)=0,\ R_{4}(e_3)=a_{3}e_3,~~h\neq 0,~~k=0.$
\item[]  $R_{5}(e_1) =0,\ R_{5}(e_2)=a_{2}e_1,\ R_{5}(e_3)=0,~~h=-\frac{1}{2}.$
\item[]  $R_{6}(e_1) =0,\ R_{6}(e_2)=a_{2}e_1,\ R_{6}(e_3)=a_{3}e_3,~~h\neq -\frac{1}{2},~~k=0.$
\item[]  $R_{7}(e_1) = 0,\ R_{7}(e_2)=0,\ R_{7}(e_3)=0,~~h\neq -\frac{1}{2},0.$
  \end{itemize}
\subsubsection{\textbf{Rota-Baxter operators on pre-Lie superalgebras of type }$\widehat{A}_{8}$}$ $

$\blacksquare$ \textsf{The pre-Lie superalgebra $\Big((\widehat{A}_{8,1})_k,\circ\Big)\simeq (D_4)_\mu$:}
$$ ~e_1\circ e_1=2e_1,~~e_2\circ e_1=e_2,~~~e_2\circ e_2=e_1,~~e_3\circ e_1=ke_3. $$
Rota-Baxter operators $RB((\widehat{A}_{8,1})_k)$ are:
\begin{itemize}
 \item[] $R_{1}(e_1)= 0,\ R_{1}(e_2)=0,\ R_{1}(e_3)=a_{1}e_3,~~k=0.$
\item[]  $R_{2}(e_1) = 0,\ R_{2}(e_2)=0,\ R_{2}(e_3)=0.$
  \end{itemize}
\subsubsection{\textbf{Rota-Baxter operators on pre-Lie superalgebras of type }$\widehat{A}_{9}$}$ $

$\blacksquare$ \textsf{The pre-Lie superalgebra $\Big((\widehat{A}_{9,1})_k,\circ\Big)\simeq (D_4)_\mu$:} $k=0~~or~~1$ associative others cases are non-associative.
$$ ~e_2\circ e_1=e_1,~~e_2\circ e_2= e_2,~~~e_3\circ e_2=ke_3. $$
Rota-Baxter operators $RB((\widehat{A}_{9,1})_k)$ are:\\
$\triangleright$ \textbf{Case 1}: If $k=0$, we have
\begin{itemize}
\item[]  $R_{1}(e_1) =a_{1}e_1+a_{2}e_2,\ R_{1}(e_2)= -\frac{a_{1}^2}{a_{2}}e_1-a_{1}e_2,\ R_{1}(e_3)=a_{3}e_3,~~a_{2}\neq 0.$
\item[]  $R_{2}(e_1) =0,\ R_{2}(e_2)=a_{4}e_1,\ R_{2}(e_3)=a_{3}e_3.$
\item[]  $R_{3}(e_1) = 0,\ R_{3}(e_2)=0,\ R_{3}(e_3)=a_{3}e_3.$
  \end{itemize}
$\triangleright$ \textbf{Case 2}: If $k\in \mathbb{C}^\ast$, we have
\begin{itemize}
 \item[] $R_{4}(e_1) = 0,\ R_{4}(e_2)=a_{4}e_1,\ R_{4}(e_3)=0.$
\item[]  $R_{5}(e_1) = a_{1}e_1+a_{2}e_2,\ R_{5}(e_2)=-\frac{a_{1}^2}{a_{2}}e_1-a_{1}e_2,\ R_{5}(e_3)= 0,~~a_{2}\neq 0.$
\item[]  $R_{6}(e_1) = 0,\ R_{6}(e_2)=0,\ R_{6}(e_3)=0.$
  \end{itemize}
$\blacksquare$ \textsf{The pre-Lie superalgebras $\Big((\widehat{A}_{9,2})_k,\circ\Big)\simeq (D_4)_\mu$} and $(\widehat{A}_{9,3},\circ)\simeq D_5$:\\
$\Big((\widehat{A}_{9,2})_k,\circ\Big) \left\{
                                                  \begin{array}{ll}
                                                    e_2\circ e_1=e_1 & \hbox{} \\
                                                    e_2\circ e_2= e_2 & \hbox{} \\
                                                    e_2\circ e_3=e_3 & \hbox{} \\
                                                    e_3\circ e_1=ke_3 & \hbox{.}
                                                  \end{array}
                                                \right.
(\widehat{A}_{9,3},\circ) \left\{
                                                        \begin{array}{ll}
                                                          e_2\circ e_2=e_1 & \hbox{} \\
                                                          e_2\circ e_2= e_2 & \hbox{} \\
                                                          e_2\circ e_3=e_3 & \hbox{} \\
                                                          e_3\circ e_3=\frac{1}{2}e_1 & \hbox{} \\
                                                          e_3\circ e_3=e_1 & \hbox{.}
                                                        \end{array}
                                                      \right.$\\
 $ k=0~~or~~1$ ~~associative~~other ~~cases~~ are~~non~~associative.\\
They have the same Rota-Baxter operators $RB(\mathcal{A})$, that is,
\begin{itemize}
\item[]  $R_{1}(e_1) = a_{1}e_1+a_{2}e_2,\ R_{1}(e_2)=-\frac{a_{1}^2}{a_{2}}e_1-a_{1}e_2,\ R_{1}(e_3)= 0,~~a_{2}\neq 0.$
 \item[] $R_{2}(e_1) = 0,\ R_{2}(e_2)=a_{3}e_1,\ R_{2}(e_3)=0.$
\item[]  $R_{3}(e_1) = 0,\ R_{3}(e_2)=0,\ R_{3}(e_3)=0.$
  \end{itemize}
\subsubsection{\textbf{Rota-Baxter operators on pre-Lie superalgebras of type }$\widehat{A}_{10_{h}}$}$ $

$\blacksquare$ \textsf{The pre-Lie superalgebra $\Big((\widehat{A}_{10_{h},1})_k,\circ\Big)\simeq (D_4)_\mu$:}
$$ e_1\circ e_2=(h-1)e_1,~~e_2\circ e_1=h e_1,~~ e_2\circ e_2= e_1+h e_2 ,~~e_3\circ e_2=k e_3,~~~~h\neq 0.$$
Rota-Baxter operators $RB((\widehat{A}_{10_{h},1})_k)$ are:\\
$\triangleright$ \textbf{Case 1}: If $k=0$, we have
\begin{itemize}
\item[]  $R_{1}(e_1) =0,\ R_{1}(e_2)=a_{1}e_1,\ R_{1}(e_3)=a_{2}e_3.$
 \item[] $R_{2}(e_1) =0,\ R_{2}(e_2)=0,\ R_{2}(e_3)=a_{2}e_3.$
  \end{itemize}
$\triangleright$ \textbf{Case 2}: If $k\in \mathbb{C}^\ast$, we have
\begin{itemize}
\item[]  $R_{3}(e_1) =0,\ R_{3}(e_2)=a_{1}e_1,\ R_{3}(e_3)=0,~~a_{1}\neq 0.$
\item[]  $R_{4}(e_1) =0,\ R_{4}(e_2)=0,\ R_{4}(e_3)=0.$
  \end{itemize}
$\blacksquare$ \textsf{The pre-Lie superalgebras $\Big((\widehat{A}_{10_{h},2})_k,\circ\Big)\simeq (D_4)_\mu$} and $(\widehat{A}_{10_{h},3},\circ)\simeq D_5$, where\\
$\Big((\widehat{A}_{10_{h},2})_k,\circ\Big) \left\{
                                                        \begin{array}{ll}
                                                          e_1\circ e_2=(h-1)e_1 & \hbox{} \\
                                                          e_2\circ e_1=h e_1 & \hbox{} \\
                                                          e_2\circ e_2= e_1+h e_2,~~~~h\neq 0 & \hbox{} \\
                                                          e_2\circ e_3=h e_3 & \hbox{} \\
                                                          e_3\circ e_2=k e_3 & \hbox{}
                                                        \end{array}
                                                      \right.
(\widehat{A}_{10_{h},3},\circ)
\left\{
  \begin{array}{ll}
    e_1\circ e_2=(h-1)e_1 & \hbox{} \\
    e_2\circ e_1=h e_1 & \hbox{} \\
    e_2\circ e_2= e_1+h e_2,~~~~h\neq 0 & \hbox{} \\
    e_2\circ e_3=h e_3 & \hbox{} \\
    e_3\circ e_2=(h-\frac{1}{2}) e_3 & \hbox{} \\
    e_3\circ e_3=e_1 & \hbox{.}
  \end{array}
\right.
$\\
have the same Rota-Baxter operators, that is,
\begin{itemize}
\item[]  $R_{1}(e_1) =0,\ R_{1}(e_2)=a_{1}e_1,\ R_{1}(e_3)=0,~~a_{1}\neq 0.$
\item[]  $R_{2}(e_1) =0 ,\ R_{2}(e_2)=0,\ R_{2}(e_3)= 0.$
  \end{itemize}
\subsubsection{\textbf{Rota-Baxter operators on pre-Lie superalgebras of type }$\widehat{A}_{11}$}$ $

$\blacksquare$ \textsf{The pre-Lie superalgebra $\Big((\widehat{A}_{11,1})_k,\circ\Big)\simeq (D_4)_\mu$:}
$$ e_1\circ e_2=-e_1,~~e_2\circ e_2=e_1-e_2,~~e_3\circ e_2=ke_3. $$
Rota-Baxter operators $RB((\widehat{A}_{11,1})_k)$ are:\\
$\triangleright$ \textbf{Case 1}: If $k=0$, we have
\begin{itemize}
\item[]  $R_{1}(e_1) = 0,\ R_{1}(e_2)=a_{1}e_1,\ R_{1}(e_3)=a_{2}e_3,~~a_{1}\neq 0.$
\item[]  $R_{2}(e_1) = 0,\ R_{2}(e_2)=0,\ R_{2}(e_3)=a_{2}e_3.$
  \end{itemize}
$\triangleright$ \textbf{Case 2}: If $k\in \mathbb{C}^\ast$, we have
\begin{itemize}
\item[]  $R_{3}(e_1) = 0,\ R_{3}(e_2)=a_{1}e_1,\ R_{3}(e_3)=0,~~a_{1}\neq 0.$
\item[]  $R_{4}(e_1) = 0,\ R_{4}(e_2)=0,\ R_{4}(e_3)=0.$
  \end{itemize}
$\blacksquare$ \textsf{The pre-Lie superalgebras $(\widehat{A}_{11,2},\circ)\simeq D_5$} and $\Big((\widehat{A}_{11,3})_k,\circ\Big)\simeq (D_4)_\mu$, where
\begin{itemize}
\item[] \ \ $(\widehat{A}_{11,2},\circ)\ \ \ : \ \ \ e_1\circ e_2=-e_1,~~e_2\circ e_2=e_1-e_2,~~e_3\circ e_2=-\frac{1}{2} e_3,~~~~e_3 \circ e_3=e_1.$
\item[] $\Big((\widehat{A}_{11,3})_k,\circ\Big): \ \ \ e_1\circ e_2=-e_1,~~e_2\circ e_2=e_1-e_2,~~e_2\circ e_3=-e_3,~~e_3\circ e_2=k e_3.$
\end{itemize}
 They have the same Rota-Baxter operators, that is,
\begin{itemize}
 \item[] $R_{1}(e_1) = 0,\ R_{1}(e_2)=a_{1}e_1,\ R_{1}(e_3)=0,~~a_{1}\neq 0.$
 \item[] $R_{2}(e_1) = 0,\ R_{2}(e_2)=0,\ R_{2}(e_3)=0.$
  \end{itemize}
\begin{rem}Using the above classification and Corollary \ref{RB-pre==L-den}, one may construct the $2$ and $3$-dimensional $L$-dendriform superalgebras associated to the Rota-Baxter pre-Lie superalgebras of dimension $2$ and $3$ $($of weight zero$)$ described above.
\end{rem}

\paragraph{\textbf{Acknowledgment.} We would like to  thank  Chengming Bai for his valuable remarks and suggestions. }

\end{document}